\documentclass{amsart}
\usepackage{amsmath}
\usepackage{amssymb}
\usepackage{mathtools}
\usepackage{MnSymbol}
\usepackage[shortlabels]{enumitem}
\usepackage{tikz-cd}
\usepackage{extarrows}
\usepackage{makecell}
\usepackage[OT2,T1]{fontenc}

\usepackage[
        colorlinks, citecolor=darkgreen,
        backref,
        pdfauthor={Samir Siksek},
]{hyperref}

\usepackage{comment}

\newcommand{\F}{\mathbb{F}}

\newcommand{\G}{\mathbb{G}}
\newcommand{\PP}{\mathbb{P}}
\newcommand{\Q}{\mathbb{Q}}
\newcommand{\R}{\mathbb{R}}
\newcommand{\Z}{\mathbb{Z}}
\newcommand{\fq}{\mathfrak{q}}

\newcommand{\cN}{\mathcal{N}}

\newcommand{\ff}{\mathfrak{f}}
\newcommand{\fp}{\mathfrak{p}}

\newcommand{\OO}{\mathcal{O}}

\DeclareMathOperator{\PGL}{PGL}

\DeclareMathOperator{\Aut}{Aut}

\DeclareMathOperator{\Gal}{Gal}

\DeclareMathOperator{\Trace}{Trace}
\DeclareMathOperator{\Norm}{Norm}

\DeclareMathOperator{\ord}{ord}

\DeclareMathOperator{\un}{un}

\renewcommand{\setminus}{-}

\newcommand{\Gisom}{\overline{\rho}_{E,p}\sim\overline{\rho}_{\mathfrak{f},\varpi}}
\newcommand{\modpg}{\overline{\rho}_{E,p}}
\newcommand{\GL}{\operatorname{GL}}
\newcommand{\SL}{\operatorname{SL}}
\newcommand{\vv}{\mathbf{v}}

\newtheorem{thm}{Theorem}

\newtheorem{lem}[thm]{Lemma}
\newtheorem{conj}[thm]{Conjecture}
\newtheorem{cor}[thm]{Corollary}
\newtheorem{prop}[thm]{Proposition}

\theoremstyle{definition}
\newtheorem{example}[thm]{Example}
\newtheorem{exercise}[thm]{Exercise}
\newtheorem{definition}[equation]{Definition}

\theoremstyle{remark}

\definecolor{darkgreen}{rgb}{0,0.5,0}

\begin{document}

\title[]{
The modular approach to Diophantine 
equations over totally real fields
}
\author{Maleeha Khawaja}

\address{
	School of Mathematics and Statistics\\
	Hicks Building\\
	University of Sheffield\\
	Sheffield S3 7RH\\
	United Kingdom
	}
\email{mkhawaja2@sheffield.ac.uk}

\author{Samir Siksek}

\address{Mathematics Institute\\
    University of Warwick\\
    CV4 7AL \\
    United Kingdom}

\email{s.siksek@warwick.ac.uk}
\date{\today}
\thanks{
Khawaja is supported by an EPSRC studentship from the University of Sheffield (EP/T517835/1).
Siksek is supported by the
EPSRC grant \emph{Moduli of Elliptic curves and Classical Diophantine Problems}
(EP/S031537/1). }
\makeatletter
\@namedef{subjclassname@2020}{%
  \textup{2020} Mathematics Subject Classification}
\makeatother

\maketitle

\begin{abstract}
Wiles' proof of Fermat's last theorem initiated a powerful new approach towards the resolution of certain Diophantine equations over $\Q$.  Numerous novel obstacles arise when extending this approach to the resolution of Diophantine equations over totally real number fields.  We give an extensive overview of these obstacles as well as providing a survey of existing methods and results in this area. 
\end{abstract}

\tableofcontents

{
  \hypersetup{linkcolor=black}
  \tableofcontents
}

\part{Introduction}
Every mathematician knows the statement of Fermat's Last Theorem,
proved by Wiles \cite{Wiles}.
\begin{thm}[Wiles]\label{thm:Fermat}
Let $n\geq 3$ be an integer. 
If $(a,b,c)\in \Q^3$ is a solution 
to the Fermat equation 
\begin{equation}
	\label{eq:Fermat1}
x^n+y^n=z^n
\end{equation}
then $abc=0$.
\end{thm}
Wiles' proof of Fermat's Last Theorem
rests on three major pillars:
\begin{enumerate}[(i)]
\item Mazur's isogeny theorem \cite{Mazur78};
\item Ribet's level-lowering theorem \cite{Ribet};
\item the modularity of semi-stable elliptic curves
over $\Q$, established by Wiles \cite{Wiles}
and Taylor--Wiles \cite{TaylorWiles}.
\end{enumerate}
The set of techniques used in this proof were later 
coined as the \lq\lq modular approach to Diophantine equations\rq\rq. 

\medskip

Let $E$ be an elliptic curve over $\Q$.
This can be given by a smooth integral Weierstrass model,
\[
E \; : \; Y^2+a_1 XY+a_3 Y \, = \, X^3+a_2 X^2 + a_4 X+a_6,
\qquad a_1,\dotsc,a_6 \in \Z.
\]
Write $\Delta$ for the minimal discriminant of $E$,
and $N$ for its conductor.
The modularity theorem \cite{Taylor01}
extends (iii) to all elliptic curves over
the rationals.
It asserts,
rather simply, that all elliptic curves over the rationals
are modular. Here is a more precise statement.
\begin{thm}[Wiles, Breuil, Conrad, Diamond and Taylor]\label{thm:modularity}
Let $E$ be an elliptic curve over $\Q$ of conductor $N$.
Then there is a normalized 
classical newform $f$ of weight $2$
level $N$ and rational Hecke eigenvalues 
such that $L(E,s)=L(f,s)$.
\end{thm}
Here $L(E,s)$ denotes the Hasse--Weil $L$-function
of $E$, and $L(f,s)$ denotes the Hecke $L$-function
of $f$. Another way of expressing the relationship
between $E$ and $f$ is via the $p$-adic
representations of both. Let $p$ be a prime.
Write $G_\Q=\Gal(\overline{\Q}/\Q)$.
Then, one can construct continuous representations
\[
\rho_{E,p} \; : \; G_\Q \rightarrow \GL_2(\Z_p),
\qquad
\rho_{f,p} \; : \; G_\Q \rightarrow \GL_2(\Z_p).
\]
Then the relationship $L(E,s)=L(f,s)$
may be re-expressed as 
\begin{equation}\label{eqn:relpadic}
\rho_{E,p} \sim \rho_{f,p}.
\end{equation}

Here is a more elementary way of expressing the 
relationship between $E$ and $f$. We may write
the $q$-expansion of $f$ as
\begin{equation}\label{eqn:qexp}
	f=q+\sum_{n=2}^{\infty}c_{n}q^{n}. 
\end{equation}
The relationship $L(E,s)=L(f,s)$ can be rewritten
as
$a_{p}(E)=c_p$
for all primes $p \nmid N$, where
\[
a_p(E) \; = \; p+1-\# E(\F_p).
\]

\medskip

We give a simplified sketch of 
Wiles' proof. Let $p\geq 5$ be a prime. 
Suppose $(a,b,c)\in\Q^3$ 
is a solution to \eqref{eq:Fermat1} with $abc\neq0$.
After suitably scaling and permuting $(a,b,c)$
we can suppose that
\begin{equation}\label{eqn:cond}
(a,b,c) \in \Z^3, \quad \gcd(a,b,c)=1, \quad 2 \mid b,
\qquad a^p \equiv -1 \pmod{4}.
\end{equation}
The main steps are as follows.
\begin{itemize}
\item \textbf{Frey curve.} 
Hellegouarch \cite{Hellegouarch} and Frey 
associate an elliptic curve, commonly known
as a \textbf{Frey curve},
\[
E: Y^2=X(X-a^p)(X+b^p)
\]
 to such a solution. The conditions in \eqref{eqn:cond}
ensure that $E$ is semistable.
The motivation for considering the Frey curve $E$
is that the mod $p$ representation 
$\overline{\rho}_{E,p}$ attached to the $p$-torsion of $E$
is unramified away from $p$ and $2$.
\item \textbf{Irreducibility.} 
It follows from Mazur's isogeny theorem \cite{Mazur78} that 
$\overline{\rho}_{E,p}$ is irreducible.
\item \textbf{Modularity.}
The elliptic curve $E$ is modular by Wiles.
It follows from this that $\overline{\rho}_{E,p}$
is modular, in the sense that
\[
\overline{\rho}_{E,p} \sim \overline{\rho}_{f,p}
\]
where $f$ is the newform attached to $E$
by the modularity theorem. This follows
by reducing \eqref{eqn:relpadic} modulo $p$.
\item \textbf{Level lowering.} 
Ribet's level-lowering theorem \cite{Ribet} 
implies that $\overline{\rho}_{E,p}$ is 
associated to a newform of weight $2$ and 
level $2$.
\item \textbf{Elimination.} There are no
newforms of weight $2$ and level $2$ contradicting 
the existence of $(a,b,c)$. 
\end{itemize}

There are two natural questions that arise here. 
Can the modular approach be extended to resolve 
other Diophantine equations over $\Q$? 
Can the modular approach be extended to resolve 
Diophantine equations over number fields? 
Indeed the answer to both of these questions is yes (see e.g. 
\cite{Alfaraj23}, \cite{Bennett21}, \cite{BMS23}, 
\cite{BPS19},  \cite{BCDF19}, \cite{DarmonI}, \cite{DarmonII}, 
\cite{DGMP23}, 
\cite{FNS20},
\cite{FLTsmall}, 
\cite{Frazer04}, 
 \cite{Fermat23}, 
 \cite{Koutsianas20}, 
\cite{Kraus19},
 \cite{MPT23},  
 \cite{PhilippeFermat},
   \cite{Philippe22}, 
   \cite{PhilippeDec}, \cite{SillimanVogt}).
In this survey, we give an overview of completed 
work that answers these questions in the affirmative 
as well as the challenges that 
are present in doing so.

\medskip

We restrict our exposition to totally real number fields 
since significantly less is known for more general number fields. 
For results established in the direction of more 
general number fields, we refer the 
reader to the work of \c{S}eng\"{u}n and Siksek \cite{SS18}
and of \c{T}urca\c{s} \cite{Turcas18, Turcas20},
which attack the Fermat equation \eqref{eq:Fermat1}
over certain imaginary quadratic fields,
assuming certain standard conjectures,
most notably Serre's modularity conjecture
over number fields.
The following recent breakthrough of 
Caraiani and Newton \cite{CN23} gives some
hope that the theorems of \c{S}eng\"{u}n, Siksek
and \c{T}urca\c{s} might eventually
be made unconditional. However, there is as of 
yet no analogue to Ribet's level lowering theorem
in the general number field setting. 

\begin{thm}[Caraiani and Newton]
Let $F$ be an imaginary quadratic field such that 
the elliptic curve $X_{0}(15)$ has rank $0$ over $F$. 
Let $E$ be an elliptic curve over $F$. 
Then $E$ is modular.
\end{thm}
 
Freitas and Siksek \cite{FS15} 
developed a variant of the modular approach 
whereby they associate a putative solution 
of a Fermat-type equation to a solution 
of a certain $S$-unit equation. 
We refer the reader to \cite{OS22} for a comprehensive 
survey of this approach, which has since been used in several works to obtain asymptotic results concerning Fermat-type equations (see e.g. \cite{SS18}, 
\cite{IKO20}, \cite{IKO23}, \cite{KO20}, \cite{Mocanu22}, \cite{Mocanu23}, \cite{Villagra23}).

\part{Level lowering and the Frey curve $E$}

\section{Level lowering}
We do not state Ribet's level-lowering theorem \cite{Ribet}
in complete generality. Instead we state a consequence
obtained as result of specialising the theorem
to mod $p$ Galois representations of elliptic curves.
\begin{thm}[Ribet]	
	\label{thm:Ribet}
Let $p\geq 5$ be a prime. 
Let $E$ be an elliptic curve over $\Q$ with conductor $N$ 
and minimal discriminant $\Delta$. 
Suppose $\modpg$ is irreducible. 
Let
\[
M_{p}=\prod_{\substack{q\; \mid\mid \; N \\
p \; \mid \; \ord_{q}(\Delta)}} q, \qquad
N_{p}=\frac{N}{M_{p}}. 
\]
Then $\modpg\sim \overline{\rho}_{f, \varpi}$ where $f$ 
is a newform of weight $2$ and level $N_p$. 
\end{thm}
Here $\varpi$ is a prime ideal above $p$ in the 
ring of integers $\OO_{\Q_f}$
of the Hecke eigenfield $\Q_f$ of $f$.
If $f$ is given by the $q$-expansion in \eqref{eqn:qexp},
then $\Q_f=\Q(c_2,c_3,c_4,\dotsc)$, 
the number field obtained
by adjoining to $\Q$ the Hecke eigenvalues $c_2,c_3,\dotsc$
of $f$. 
We point out that modularity of $\overline{\rho}_{E,p}$ 
is a hypothesis of the original 
version of Ribet's theorem. We do not
however need this hypothesis anymore as it is 
now ensured
by the modularity theorem (Theorem~\ref{thm:modularity}).

As observed by Freitas and Siksek \cite[Theorem 7]{FS15},
the analogous result 
over totally real fields 
is a consequence of the work of Fujiwara \cite{Fujiwara}, 
Jarvis \cite{FrazerMod} 
and Rajaei \cite{Rajaei}.
\begin{thm}[Jarvis, Fujiwara and Rajaei]\label{thm:levellowering}
Let $p\geq 5$ be a prime.
Let $K$ be a totally real field such that $K\nsubseteq\mathbb{Q}(\zeta_{p})^{+}$. 
Let $E$ be an elliptic curve over $K$ with conductor
$\mathcal{N}$. Suppose $E$ is modular and $\overline{\rho}_{E,p}$ is
irreducible.  
For a prime ideal $\fq$ of $\OO_K$,
let $\Delta_{\mathfrak{q}}$ denote
the discriminant
of local minimal model of $E$ at $\mathfrak{q}$.
Suppose the following conditions are satisfied for all prime
ideals 
$\fp$ that lie over $p$: 
\begin{enumerate}[(i)]
\item 
$E$ is semistable at $\fp$;
\item $p \mid v_{\fp}(\Delta_\fp)$;
\item the ramification index satisfies $e(\fp/p)<p-1$.
\end{enumerate}
Let
\begin{equation}\label{eqn:level}
    \mathcal{M}_{p}= 
    \prod_{
    \substack{\mathfrak{q}\; \mid\mid \; \mathcal{N},\\ p \; \mid \; v_{\mathfrak{q}}(\Delta_{\mathfrak{q}})
    }
    }\mathfrak{q},\qquad 
    \mathcal{N}_{p} = \dfrac{\mathcal{N}}{M_{p}}.
\end{equation}
Then, $\Gisom$, where $\mathfrak{f}$ is a Hilbert eigenform of
level $\mathcal{N}_{p}$ and parallel weight $2$ and $\varpi$ is a
prime ideal of $\mathbb{Q}_{\mathfrak{f}}$ that lies above $p$.
\end{thm}

We give a brief outline 
of how to resolve a Diophantine equation 
over a totally real field with
the help of Theorem \ref{thm:levellowering}. 
\begin{itemize}
	\item \textbf{Select a Frey curve $E$.}
	Let $p\geq 5$ be a prime. 
	Let $E$ be an elliptic curve over $K\nsubseteq\mathbb{Q}(\zeta_{p})^{+}$ associated 
	to a putative solution to the Diophantine equation in question.	
	\item \textbf{Determine the reduction type of $E$.} 
	One needs to show that $E$ is semistable at 
	all primes of $K$ above $p$ as well as compute the conductor of $E$. 
	It is often the case that $E$ has semistable 
	reduction away from a small set of primes. 
	For example if $E$ is the Frey curve associated to a 
	putative solution to the Fermat equation over $K$ 
	then $E$ has semistable reduction away from $2$ and a 
	finite set of primes representing the class group of $K$ (see \cite[Lemma 3.3]{FS15}).
	\item \textbf{Prove $\modpg$ is irreducible.} 
	In Part \ref{part:irred} we summarise existing 
	techniques used to show that $\modpg$ is irreducible.
	\item \textbf{Prove $E$ is modular.} 
	In Part \ref{sec:modularity}, we give an overview of 
	the definition of $\modpg$, and summarise known results surrounding the modularity 
	of elliptic curves over totally real fields.
	\item \textbf{Eliminate the isomorphisms.} The final step in reaching a contradiction to the putative solution is to 
	eliminate the isomorphisms $\modpg\sim\overline{\rho}_{\mathfrak{f}, \varpi}$. 
	We give an overview of the techniques used in this step in Part \ref{sec:Elimination}.
\end{itemize}

\section{The Frey curve} The initial obstacle to 
the resolution of a 
Diophantine equation using the modular approach 
is the construction of an appropriate Frey curve. 
Let $E$ be an elliptic curve over a 
totally real field $K$ with discriminant $\Delta$ 
and conductor $\mathcal{N}$. 
Theorem \ref{thm:levellowering} 
reduces $\mathcal{N}$ by the factor $\mathcal{M}_{p}$, 
where $\mathcal{M}_{p}$ is the product of 
the primes $\mathfrak{q}$ of $K$ at which $E$ is semistable 
and $p\mid v_{\mathfrak{q}}(\Delta_{q})$. 
Thus it is critical that the $p^{th}$-power 
free part of $\Delta$ is independent of the 
putative solution to the Diophantine equation. 

\medskip
The overall picture is brighter for 
generalized Fermat equations $Ax^p+By^q=Cz^r$ of 
signature $(p,q,r)$. 
Namely, recipes for Frey elliptic 
curves have been developed 
for: 
\begin{itemize}
	\item the signature $(p,p,p)$ by Kraus \cite{Kraus97};
	\item the signature $(p,p,2)$ by Bennett and Skinner \cite{BennettSkinner};
	\item the signature $(p,p,3)$ by Bennett, Vatsal and Yazdani \cite{BVY04}.
\end{itemize}

\section{The $(p,p,p)$ Frey curve}
We start with the equation
\[
u+v+w=0.
\]
Here $u$, $v$, $w$ are non-zero integers in a number field $K$.
We consider the elliptic curve
\[
Y^2=X(X-u)(X+v).
\]
This model has discriminant
\[
\Delta=16u^2 v^2 (u+v)^2=16 u^2 v^2 w^2.
\]
The $c$-invariants of the model are
\[
c_4=16(u^2-vw)=16(v^2-wu)=16(w^2-uv)
\]
and
\[
c_6=-32(u-v)(v-w)(w-u).
\]
Note that if $\fq \nmid 2$ is a prime ideal of $\OO_K$
and $\fq \nmid uvw$ then $E$ has good reduction at $\fq$.
If $\fq \nmid 2$, and divides exactly one of $u$, $v$, $w$,
then $E$ is minimal and has multiplicative reduction at $\fq$,
since $\fq \mid \Delta$ but $\fq \nmid c_4$.

We now consider the equation
\[
A x^p+By^p+Cz^p=0
\]
of signature $(p,p,p)$. Here $A$, $B$, $C \in \OO_K$ are fixed and non-zero,
and $x$, $y$, $z \in \OO_K$ are unknown and non-zero. For simplicity,
suppose $\gcd(Ax^p,By^p,Cz^p)=1$. We apply the above
with $u=Ax^p$, $v=By^p$, $w=Cz^p$. We obtain the Frey curve
\[
E : Y^2=X(X-Ax^p)(X+By^p).
\]
The discriminant of the model $E$ is
\[
\Delta=16A^2B^2 C^2 (xyz)^{2p}.
\]
The conductor $\cN$ has the form
\[
\cN=\prod_{\fq \mid 2} \fq^{e_\fq} \cdot \prod_{\substack{\fq \mid ABCxyz \\ \fq \nmid 2}} \fq.
\]
Note that $E$ has multiplicative reduction at all $\fq \mid ABCxyz$, $\fq \nmid 2$,
so the exponent of such $\fq$ in the conductor is $1$.
The exponents $e_\fq$ for the primes above $2$ can be determined
using Tate's algorithm, but they are bounded by
$2+6 v_\fq(2)$ (\cite[Theorem IV.10.4]{SilvermanII}).
Note that $p \mid \ord_\fq(\Delta)$ for $\fq \mid xyz$, $\fq \nmid 2 ABC$.
Thus, the level $\cN$ in Theorem~\ref{thm:levellowering}
satisfies
\[
\cN=\prod_{\fq \mid 2} \fq^{e_\fq^\prime} \cdot \prod_{\substack{\fq \mid ABC \\ \fq \nmid 2}} \fq.
\]
Here $e_\fq^\prime \le e_\fq$. We note that $\cN$ essentially
depends only on the coefficients $A$, $B$, $C$. Only the exponents
of the primes above $2$ depend on the solution $(x,y,z)$.
But at any rate, there are only finitely many possibilities for $\cN$.

The simplifying assumption $\gcd(Ax^p,By^p,Cz^p)=1$ cannot always
be achieved over number fields, since the $\gcd$ is an ideal
that may not be principal. In the general case we need
to scale $x$, $y$, $z$ so that they remain integral,
but the $\gcd$ belongs to a finite set depending
on the class group. For an explanation of this,
see \cite[Lemma 3.2]{FS15} or \cite[Lemma 12]{NowrooziSiksek}.

\section{The $(p,p,2)$ and $(p,p,3)$ Frey curves}
\begin{exercise}
Consider the equation
\[
u+v=w^2
\]
and the model
\[
E \; : \; Y^2=X(X^2+2wX+u).
\]
Note that the discriminant of this model is
\[
\Delta=16 u^2((2w)^2-4u)=64 u^2v.
\]
Use this to workout the details of the Frey curve
for the $(p,p,2)$ equation
\[
Ax^p+B y^p=Cz^2.
\]
\end{exercise}
\begin{exercise}
Consider the equation
\[
u+v=w^3
\]
and the model
\[
E \; : \; Y^2+3wXY+uY=X^3.
\]
Note that the discriminant of this model is
\[
\Delta=27u^3(w^3-u)=27 u^3 v.
\]
Use this to workout the details of the Frey curve
for the $(p,p,3)$ equation
\[
Ax^p+B y^p=Cz^3.
\]
\end{exercise}

In favourable settings it maybe  possible to 
reduce the resolution of an arbitrary 
Diophantine equation to the 
resolution of several Fermat-type 
equations via descent. 

\begin{example}
Bennett, Patel and Siksek \cite{BPS16} 
resolve the question 
of when the sum of $3$ consecutive integral $5^{th}$ powers 
is equal to a perfect power. 
This can be reframed as the problem of
finding all integral solutions to the equation 
\begin{equation}
	\label{eq:fifthpowers}
	x(3x^4+20x^2+10)=z^p,
\end{equation}
with prime exponent $p$. This does not fit into the
three families above. However, let $\alpha=\gcd(x,10)$.
Then $\alpha=1$, $2$, $5$ or $10$. We note that
\[
x=\alpha^{p-1} z_1^p, \qquad 3x^4+20x^2+10=\alpha z_2^p,
\]
where $z_1$, $z_2$ are integers and $z=\alpha z_1 z_2$.
They made use of the identity,
\[
7x^4 + (3x^4+20x^2+10) \; = \; 10(x^2+1)^2,
\]
whence
\[
7 \alpha^{4p-5} z_1^{4p}+ z_2^p= \frac{10}{\alpha} (x^2+1)^2.
\]
Note that this equation fits into the $(p,p,2)$
family, and leads to a $(p,p,2)$ Frey curve.
Similar factorisation arguments are used
in \cite{BPS17} to reduce Diophantine
problems associated to sums of consecutive cubes
to ternary Diophantine equations of signature $(p,p,2)$.
\end{example}

\begin{example}
In \cite{BPS19}, the authors resolve the 
equation $F_n+2=y^p$ where $F_n$ denotes the $n$-th
Fibonacci number. The case where $n$ is odd
reduces in an elementary manner
 to the equation $F_n=y^p$ which had been resolved
by Bugeaud, Mignotte and Siksek \cite{FibLuc}.
Thus it remains to consider
\begin{equation}\label{eqn:Fnp2}
F_{2n}+2=y^p.
\end{equation}
There does not seem to be a way of constructing a Frey
curve associated to \eqref{eqn:Fnp2} over $\Q$.
However, we may write
\[
F_{2n}=\frac{\varepsilon^{2n}-\varepsilon^{-2n}}{\sqrt{5}},
\qquad \varepsilon=\frac{1+\sqrt{5}}{2}.
\]
Let $x=\varepsilon^{2n}+\sqrt{5}$. Then \eqref{eqn:Fnp2}
maybe rewritten as
\[
x^2-6 \; = \; \sqrt{5} \varepsilon^{2n} y^p
\]
which does fit in to the $(p,p,2)$ family.
\end{example}

There are numerous 
works that study integer solutions 
to various Diophantine equations 
where the associated Frey curve 
is defined over a totally real field; 
see e.g. \cite{Alfaraj23}, \cite{Anni}, 
\cite{BDSS15}, \cite{BMS23}, \cite{BPS19}, \cite{BCDF19}, \cite{DarmonI}, 
\cite{DarmonII}, \cite{DF13}, \cite{FreitasFrey}, \cite{Pacetti22}.
It appears that the origins of this 
method lie in \cite{BDSS15} wherein 
the authors construct a Frey curve 
defined over $\Q(\sqrt{2})$ in order 
to provide a complete resolution 
of the equation $x^{2n}\pm 6x^n+1=8y^2$ 
in positive integers.

\section{The multi-Frey approach}

The multi-Frey approach, introduced by Bugeaud, Mignotte
and Siksek \cite{BMS08}, increases the chances of 
success in the final step of the modular approach. 
This is when a single Diophantine
equation is attacked simultaneously 
using multiple Frey curves.  
As a very basic example, an equation of the $ax^p+by^p=c$
can trivially be written in the following three ways
\[
ax^p+by^p=c \cdot 1^p, \qquad
ax^p+by^p=c \cdot 1^2, \qquad
ax^p+by^p=c \cdot 1^3,
\]
allowing us to attach the $(p,p,p)$, $(p,p,2)$
and $(p,p,3)$ Frey curves to the single
equation $ax^p+by^p=c$.
More recently this approach was used by 
Billerey, Chen, Dieulefait and Freitas \cite[Theorem 1]{BCDF19} 
to prove the following result. 

\begin{thm}[Billerey, Chen, Dieulefait and Freitas]\label{thm:BCDF19}
Let $p$ be a prime. 
There are no integral solutions $(a,b,c)$ to the equation
\begin{equation}
	\label{eq:BCDF}
		x^5+y^5=3z^p
\end{equation}
such that $abc\neq 0$ and $\gcd(a,b,c)=1$.
\end{thm} 

Suppose $(a,b,c)$ is an integral solution 
to \eqref{eq:BCDF} with $abc\neq 0$ and $\gcd(a,b,c)=1$. 
Dieulefait and Freitas \cite{DF14} provided a resolution of \eqref{eq:BCDF} for a set 
of prime exponents with (Dirichlet) density $1/2$.
In doing so, the authors associate two Frey elliptic curves $E_{a,b}$ and $F_{a,b}$ to $(a,b,c)$ where both elliptic curves are defined over $\Q(\sqrt{5})$. 
These Frey curves are also used in 
the proof of Theorem \ref{thm:BCDF19}. 
In particular the elliptic curve $E_{a,b}$ is amenable to the modular 
approach when $5\nmid (a+b)$, and the elliptic curve 
$F_{a,b}$ is amenable to the modular approach when $5\mid (a+b)$ 
which allows for a complete resolution of \eqref{eq:BCDF}. 

We refer the reader to the following works 
for further applications of this approach:
\cite{BC12}, \cite{BCDY14}, 
\cite{BMS23}, \cite{BPS19}, \cite{DarmonII}, \cite{BLMS08}, \cite{FreitasFrey}, \cite{MPT23}. 

\section{Frey hyperelliptic curves}
Darmon \cite{Darmon00} has developed a 
program to attack Fermat-type equations through 
replacing the Frey elliptic curve with certain 
higher dimensional abelian varieties. 
Inspired by Darmon's program, Billerey, Chen, Dieulefait and 
Freitas \cite{DarmonI} study integral solutions 
$(a,b,c)$ to the generalised Fermat 
equation $x^r+y^r=Cz^p$ where $C\in\Z\setminus \{0\}$, $r\geq 5$ is a fixed 
prime and $p$ is a prime that varies. 
They prove the following result 
through the application of the modular approach 
to a Frey hyperelliptic 
curve $C_{r}(a,b)$ constructed by Kraus \cite{KrausNotes}.

\begin{thm}[Billerey, Chen, Dieulefait and Freitas]
Let $n\geq 2$ be an integer. 
There are no integral solutions $(a,b,c)$ to the equation 
\[
	x^{11}+y^{11}=z^n
\]
such that $abc\neq 0,\; \gcd(a,b,c)=1$ and 
$2\mid (a+b)$ or $11\mid (a+b)$.
\end{thm}

The authors use a level-lowering result due 
to Breuil and Diamond \cite[Theorem 3.2.2]{BD14}; 
the steps are parallel to those outlined in the introduction. 
See also \cite{DarmonII} and \cite{Chen22} 
for applications of the modular approach using 
Frey hyperelliptic curves.

\part{Galois representations associated to elliptic curves}\label{part:irred}
We shall
introduce Galois representations of elliptic curves
from scratch.
Much of the material
in this part can in fact be found in Silverman's book
\cite{SilvermanI}, but we rewrite
it in fashion that emphasizes Galois representations.
Whilst it is possible to mostly avoid
speaking about Galois representations when
applying the modular approach over $\Q$
(e.g. \cite{SiksekModular}),
this is not possible over number fields,
as we do not have an analogue of 
Mazur's isogeny theorem. We shall therefore
go through this material in some detail.

\section{Definition and first examples}
Notation:

\begin{tabular}{ccl}
$K$ & \qquad & a perfect field\\
$G_K$ & \qquad & $=\Gal(\overline{K}/K)$ the absolute Galois group
of $K$\\
$N$ & \qquad & a positive integer, not divisible by the
characteristic of $K$.\\ 
$E$ & \qquad & an elliptic curve defined over $K$.
\end{tabular}

Recall
\begin{equation}\label{eqn:E[p]}
E[N]\cong \Z/N\Z \oplus \Z/N\Z
\end{equation}
and so $E[N]$ has rank $2$ as
a  $\Z/N\Z$-module.
If $K \subseteq \mathbb{C}$ then
we may see this as follows. Recall that there is some $\tau \in \mathbb{H}$
(the upper half-plane) and a complex analytic isomorphism
\[
E(\mathbb{C}) \cong \frac{\mathbb{C}}{\Z+\tau \Z}.
\]
Thus 
\[
E(\mathbb{C}) \cong \R/\Z \times \R/\Z
\]
 from which
\eqref{eqn:E[p]} follows.
However, $E[N] \subset E(\overline{K})$
and is stable under the action of $G_K$. In fact, $G_K$
acts linearly on $E[N]$, i.e.
\[
\sigma(P+Q)=\sigma(P)+\sigma(Q), \qquad \sigma(aP)=a\sigma(P)
\]
for any $P$, $Q \in E[N]$, $a \in \Z/N\Z$, and $\sigma \in G_K$.
Thus $\sigma \in G_K$ induces an automorphism
of the $\Z/N\Z$-module $E[N]$, and we write
$\overline{\rho}_{E,N}(\sigma)$ for this automorphism:
\[
\overline{\rho}_{E,N}(\sigma) : E[N] \rightarrow E[N], \qquad
P \mapsto \sigma(P).
\]
We obtain a representation (i.e. a homomorphism)
\[
\overline{\rho}_{E,N} \; : \; G_K \rightarrow \Aut(E[N])
\]
where $\Aut(E[N])$ is the automorphism group of $E[N]$.
The representation $\overline{\rho}_{E,N}$
 is known as the \textbf{mod $N$ Galois representation} attached to $E$.
Choosing a basis for $E[N]$ we
can identify $\overline{\rho}_{E,N}$ as a representation
\[
\overline{\rho}_{E,N} \; : \; G_K \rightarrow \GL_2(\Z/N\Z).
\]
To see this identification, let $P_1$, $P_2$ be a basis for $E[N]$.
Then, for any $\sigma \in \G_K$,
\[
\sigma(P_1)=a_\sigma P_1+c\sigma P_2, \qquad \sigma(P_2)=b_\sigma P_1+d_\sigma P_2,
\qquad a_\sigma,~b_\sigma,~c_\sigma,~d_\sigma \in \Z/N\Z.
\]
We associate to $\sigma$ the $2\times 2$ matrix
\[
\overline{\rho}_{E,N}(\sigma)=\begin{pmatrix}
a_\sigma & b_\sigma\\
c_\sigma & d_\sigma
\end{pmatrix}.
\]
It is easy to check that
\[
\overline{\rho}_{E,N}(\sigma \tau)=\overline{\rho}_{E,N}(\sigma) \overline{\rho}_{E,N}(\tau)
\]
for $\sigma$, $\tau \in G_K$. Applying this with $\tau=\sigma^{-1}$
shows that $\overline{\rho}_{E,N}(\sigma)$ is invertible: i.e.
$\overline{\rho}_{E,N}(\sigma) \in \GL_2(\Z/N\Z)$. 
We obtain a representation (i.e. homomorphism)
\[
\overline{\rho}_{E,N} \; : \; G_K \rightarrow \GL_2(\Z/N\Z),
\]
which is nothing more than our ealier $\rho_{E,N}$ written
in terms of $2 \times 2$-matrices.
The above construction makes it look as if $\overline{\rho}_{E,N}$
depends on the choice of basis $P_1$, $P_2$ for $E[N]$.
However, changing the basis results in an equivalent
representation.

\medskip

Since $\GL_2(\Z/N\Z)$ is finite, we know that the
kernel $\ker(\overline{\rho})$ is normal of finite index.
Moreover,
\[
\sigma \in \ker(\overline{\rho}) \iff 
\text{$P^\sigma=P$ for all $P \in E[N]$}.
\]
Thus
\[
\ker(\overline{\rho})=G_{K(E[N])}.
\]
Hence
\[
\overline{\rho}(G_K) \; \cong \; G_K/G_{K(E[N])} \; \cong \; \Gal(K(E[N])/K).
\]

\section{An Example: $\overline{\rho}_{E,2}$}\label{sec:example}
In simple examples we can sometimes guess what the image $\overline{\rho}(G_K)$
has to be. The simplest case is when $N=2$.
Here we are supposing that the characteristic of $K$ is not $2$.
We can write
\[
E \; : \; Y^2=f(X), \qquad f(X)=X^3+a X^2+bX+c \in K[X], \quad \Delta(f) \ne 0.
\]
Recall that the points of order $2$ are $(\theta_i,0)$
where $\theta_1$, $\theta_2$, $\theta_3$ are the roots of $f$.
Write $P_i=(\theta_i,0)$. Then $P_1$, $P_2$ is a basis
for $E[2]=\{0,P_1,P_2,P_3\}$ and $P_3=P_1+P_2$.
Observe that
\[
K(E[2])=K(\theta_1,\theta_2,\theta_3), \qquad
\Gal(K(E[2])/K)=\Gal(f).
\]

\begin{itemize}
\item If $\theta_1$, $\theta_2$, $\theta_3 \in K$, then
$\overline{\rho}=1$ (the trivial homomorphism).
\item Suppose $\theta_1 \in K$, $\theta_2 \notin K$
and so $\theta_3 \notin K$.
We can write $f(X)=(X-\theta_1)(X^2+uX+v)$ where $u$, $v \in K$,
and $d=u^2-4v \in K^*\setminus (K^*)^2$. Thus $\theta_2$, $\theta_3$
are the two roots of the irreducible quadratic factor $X^2+uX+v$,
and
$K(E[2])=K(\theta_2)=K(\theta_3)=K(\sqrt{d})$.
We shall write $\overline{\rho}_{E,2}$ with respect to the basis
$P_1$, $P_2$.
Let $\sigma \in G_K$.
If $\sigma(\sqrt{d})=\sqrt{d}$ then
\[
\sigma(P_1)=P_1, \quad \sigma(P_2)=P_2,
\qquad \overline{\rho}(\sigma)=
\begin{pmatrix}
1 & 0 \\
0 & 1
\end{pmatrix} \in \GL_2(\F_2)
\]
If $\sigma(\sqrt{d})=-\sqrt{d}$ then $\sigma$ swaps $\theta_2$, $\theta_3$,
so
\[
\sigma(P_1)=P_1, \quad \sigma(P_2)=P_3=P_1+P_2,
\qquad \overline{\rho}(\sigma)=\begin{pmatrix}
1 & 1 \\
0 & 1
\end{pmatrix} \in \GL_2(\F_2).
\]
Note that
\[
\overline{\rho}(G_K)=\left\{ 
\begin{pmatrix}
1 & 0 \\
0 & 1
\end{pmatrix} , \; 
\begin{pmatrix}
1 & 1 \\
0 & 1
\end{pmatrix} 
\right\} \cong \Z/2\Z \cong \Gal(K(\sqrt{d})/K) = \Gal(K(E[2])/K).
\]
\item Suppose $f$ is irreducible, but $\Delta(f) \in (K^*)^2$.
Then $\Gal(f) \cong A_3$.
Let $\sigma \in G_K$.
Then $\sigma$ acts on $(\theta_1,\theta_2,\theta_3)$
via one of the three permutations $\mathrm{id}$, $(1,2,3)$, $(1,3,2) \in A_3$.
\[
\begin{gathered}
 (\theta_1,\theta_2,\theta_3)^\sigma=
(\theta_1,\theta_2,\theta_3) \implies
P_1^\sigma=P_1, \quad
P_2^\sigma=P_1\implies
\overline{\rho}(\sigma)=
\begin{pmatrix}
1 & 0\\
0 & 1
\end{pmatrix}\\
 (\theta_1,\theta_2,\theta_3)^\sigma=
(\theta_2,\theta_3,\theta_1) \implies
P_1^\sigma=P_2, \quad
P_2^\sigma=P_3=P_1+P_2 \implies
\overline{\rho}(\sigma)=
\begin{pmatrix}
0 & 1\\
1 & 1
\end{pmatrix}\\
 (\theta_1,\theta_2,\theta_3)^\sigma=
(\theta_3,\theta_1,\theta_2) \implies
P_1^\sigma=P_3=P_1+2, \quad
P_2^\sigma=P_1 \implies
\overline{\rho}(\sigma)=
\begin{pmatrix}
1 & 1\\
1 & 0
\end{pmatrix}.
\end{gathered}
\]
Thus
\[
\overline{\rho}(G_K)\; = \;
\left\{
\begin{pmatrix}
1 & 0\\
0 & 1
\end{pmatrix}, \;
\begin{pmatrix}
0 & 1\\
1 & 1
\end{pmatrix}, \;
\begin{pmatrix}
1 & 1\\
1 & 0
\end{pmatrix}
\right\} \cong \Z/3\Z \cong \Gal(K(E[2])/K).
\]
\item Suppose $f$ is irreducible and $\Delta(f) \notin {K^*}^2$.
Thus $\Gal(K(E[2]/K)=\Gal(f)=S_3$. Thus $\overline{\rho}(G_K)$
is a subgroup of $\GL_2(\F_2)$ that is isomorphic to $S_3$.
But $\#S_3=\#\GL_2(\F_2)=6$. Hence $\overline{\rho}$
is surjective and we also arrive at the conclusion that
$S_3 \cong \GL_2(\F_2)$. It's easy to
write the matrix $\overline{\rho}(\sigma)$ in terms
of the action of $\sigma$ on $\theta_1$, $\theta_2$, $\theta_3$.
\end{itemize}

\medskip

\noindent \textbf{Important Remark.}
The image $\overline{\rho}(G_K) \subseteq \GL_2(\Z/N\Z)$
depends on a choice of basis for $E[N]$. If we change
basis then we conjugate $\overline{\rho}$ by
the change of basis matrix, which is an element
of $\GL_2(\Z/N\Z)$. So the image is really only
well defined up to conjugation.

\section{The mod $N$-cyclotomic character}
Let $\zeta_N$ be a primitive $N$-th root of $1$. Define
\textbf{the mod $N$-cyclotomic character}
$\chi_N \; : \; G_\Q \rightarrow (\Z/N\Z)^*$
as follows. For $\sigma \in G_\Q$ we see that $\zeta_N^\sigma$
is also a primitive $N$-root of unity and so
$\zeta_N^\sigma=\zeta_N^{a_\sigma}$
where $a_\sigma$ is an integer, coprime to $N$, and whose
value is defined only modulo $N$, i.e.
$a_\sigma \in (\Z/N\Z)^*$. We let $\chi_N(\sigma)=a_\sigma$.
To summarise,
\[
\chi_N \; : \; G_K \rightarrow (\Z/N\Z)^*, \qquad \zeta_N^\sigma=\zeta_N^{\chi_N(\sigma)}.
\]

\begin{thm} \label{thm:cyclotomiccharacter}
Let $K$ be a number field.
\begin{enumerate}
\item[(i)] If $\tau \in G_K$ denotes any complex
conjugation
\footnote{Let us explain what complex conjugation
is. Let $K$ be a number field and let
$\iota_\infty : K \hookrightarrow \R$ be a real embedding of $K$.
Let $\iota : \overline{K} \hookrightarrow \mathbb{C}$ be an embedding
extending $\iota_\infty$. Let $c : \mathbb{C} \rightarrow \mathbb{C}$ denote
complex conjugation. Then $\iota^{-1} \circ c \circ \iota$
is an element of $G_K$ which we call a \textbf{complex conjugation}.
Of course if $K$ is totally complex then it does not have any complex
conjugations. You can check that the conjugacy classes of
complex conjugations inside $G_K$ are in bijection with the real
embeddings of $K$},
then
$\chi_N(\tau)=-1$.
\item[(ii)] Let $\lambda \ne N$ be finite place of $K$,
and let $I_\lambda \subset G_K$
denote an inertia subgroup at $\lambda$. Then
$\chi_N(I_\lambda)=1$ (we say that
$\chi_N$ is \textbf{unramified} at $\lambda$).
Moreover, if $\sigma_\lambda \in G_K$ denotes a Frobenius
element at $\lambda$, then
\[
\chi_N(\sigma_\lambda) \equiv \Norm_{K/\Q}(\lambda) \pmod{N}.
\]
\end{enumerate}
\end{thm}
\begin{proof}
Part (i) is clear as complex conjugation takes $\zeta_N$ to $\zeta_N^{-1}$.

We turn to (ii). Corresponding to $I_\lambda$ is a prime $\mu \mid \lambda$
of $\overline{K}$ (changing $\mu$ conjugates $I_\lambda$ and so leaves the
desired result unaffected). By definition of inertia,
\[
\zeta_N^\sigma \equiv \zeta_N \pmod{\mu}
\]
for all $\sigma \in I_\lambda$. Recall that the difference of two distinct
$N$-th roots of $1$ divides $N$. As $\lambda \ne N$ we have $\mu \nmid N$.
Thus $\zeta_N^\sigma=\zeta_N$. But $\zeta_N^\sigma=\zeta_N^{\chi_N(\sigma)}$
by definition of $\chi_N$. It follow that $\zeta_N^{\chi_N(\sigma)}=\zeta_N$
and $\chi_N(\sigma)=1$ for all $\sigma \in I_\lambda$.

Now let $\sigma_\lambda$ be a Frobenius element corresponding to $\lambda$. Then
\[
\zeta_N^{\sigma_\lambda} \equiv \zeta_N^{\Norm_{K/\Q}(\lambda)} \pmod{\mu}
\]
by definition of Frobenius.
As above $\zeta_N^{\sigma_\lambda}=\zeta_N^{\Norm_{K/\Q}(\lambda)}$.
Hence $\chi_N(\sigma_\lambda)\equiv \Norm_{K/\Q}(\lambda) \pmod{N}$.
\end{proof}

\begin{thm}\label{thm:detrep}
$\det \overline{\rho}_{E,N}=\chi_N$.
\end{thm}
\begin{proof}
Recall that the Weil pairing
\[
e_N \; : \; E[N] \times E[N] \rightarrow \mu_N=\langle \zeta_N \rangle
\]
is bilinear, alternating\footnote{As a reminder, \textbf{alternating}
means $e_N(S,S)=1$ for all $S \in E[N]$. This implies that $e_N$ is
\textbf{skew-symmetric}: $e_N(T,S)=e_N(S,T)^{-1}$. To see this note
\[
1=e_N(S+T,S+T)=e_N(S,S)e_N(S,T)e_N(T,S)e_N(T,T)=e_N(S,T)e_N(T,S).
\]
}, non-degenerate and Galois invariant.

As $e_N$ is non-degenerate, we may choose a basis $S$, $T$ for $E[N]$
such that $e_N(S,T)=\zeta_N$. Let $\sigma \in G_K$.
Write
\[
\overline{\rho}_{E,N}(\sigma)=\begin{pmatrix}
a & b \\
c & d
\end{pmatrix}.
\]
Thus
\[
S^\sigma=aS+cT, \qquad T^\sigma=bS+dT.
\]
Then
\[
\begin{split}
\zeta_N^{\chi_N(\sigma)}&=\zeta_N^{\sigma} \qquad \text{by definition of $\chi_N$}\\
&=e_N(S,T)^\sigma \qquad \text{by choice of $S$, $T$}\\
&=e_N(S^\sigma,T^\sigma) \qquad \text{by Galois invariance of $e_N$}\\
&=e_N(aS+cT,bS+dT)\\
&=e_N(S,S)^{ac} e_N(S,T)^{ad} e_N(T,S)^{bc} e_N(T,T)^{cd} \qquad
\text{by bilinearity of $e_N$}\\
&=e_N(S,T)^{ad-bc} \qquad \text{as $e_N$ is alternating}\\
&=\zeta_N^{ad-bc} \qquad \text{again by choice of $S$, $T$}.
\end{split}
\]
Hence $\chi_N(\sigma)=ad-bc=\det \overline{\rho}_{E,N}(\sigma)$
completing the proof.
\end{proof}

When $K$ is a number field we say that
a representation $\overline{\rho} : G_K \rightarrow \GL_2(\Z/N\Z)$
is \textbf{odd} if for every complex conjugation $\tau \in G_K$
we have $\det(\overline{\rho}(\tau))=-1$.

\begin{cor}
Let $E$ be an elliptic curve over a number field $K$.
Then $\overline{\rho}_{E,N}$ is odd.
\end{cor}
\begin{proof}
This follows from Theorems~\ref{thm:cyclotomiccharacter}
and \ref{thm:detrep}.
\end{proof}
Of course if $K$ is totally complex, then the corollary is vacuous.

\section{Torsion and isogenies}
\begin{thm}\label{thm:ptorsion}
The following are equivalent:
\begin{enumerate}
\item[(a)] $E$ has a $K$-rational point of order $N$;
\item[(b)]
$\displaystyle \overline{\rho}_{E,N} \sim
\begin{pmatrix}
1 & * \\
0 & \chi_N
\end{pmatrix}$.
\item[(c)] $\overline{\rho}_{E,N}(G_K)$ is conjugate inside
$\GL_2(\Z/N\Z)$ to a subgroup of
\[
B_1(N) \; := \; 
\left\{ 
\begin{pmatrix}
1 & b \\
0 & d
\end{pmatrix} \; : \; b \in \Z/N\Z, \; d \in (\Z/N\Z)^*\right\}
\subset \GL_2(\Z/N\Z).
\]
\end{enumerate}
\end{thm}
\begin{proof}
(a)$\implies$(b). Suppose $E$ has a $K$-rational point $P$ of order $N$.  Let $Q \in E[N]$
so that $P$, $Q$ is a $\Z/N\Z$-basis for $E[N]$. Then for all $\sigma \in G_K$,
we have
\[
\sigma(P)=P, \qquad \sigma(Q)=b_\sigma P+d_\sigma Q.
\]
Hence
\[
\overline{\rho}_{E,N}(\sigma)=\begin{pmatrix}
1 & b_\sigma \\
0 & d_\sigma
\end{pmatrix}
\]
for all $\sigma \in G_K$.
However, by Theorem~\ref{thm:detrep},
\[
d_\sigma=\det \overline{\rho}_{E,N}(\sigma)=\chi_N(\sigma)
\]
Thus (b) holds.

(b)$\implies$(c). This is clear.

(b)$\implies$(c). This is clear.

(c)$\implies$(a). Suppose (c). Then we can choose
a basis $P$, $Q$ so that the image $\overline{\rho}_{E,N}(G_K)$
is contained in $B_1(N)$. Note that
\[
P^\sigma=P, \qquad Q^\sigma=b_\sigma P+d_\sigma  Q
\]
for all $\sigma \in G_K$ (where $b_\sigma$, $d_\sigma \in \Z/N\Z$).
As $P$ is fixed by $G_K$ it follows that $P \in E(K)$.
Since $P$, $Q$ is a basis, $P$ must have exact order $N$,
proving (a).
\end{proof}

\begin{thm}\label{thm:pisog}
The following are equivalent:
\begin{enumerate}
\item[(a)] $E$ has a cyclic $K$-rational $N$-isogeny;
\item[(b)]
$\displaystyle \overline{\rho}_{E,N} \sim
\begin{pmatrix}
\phi & * \\
0 & \psi
\end{pmatrix}$,
where $\phi$, $\psi : G_K \rightarrow (\Z/N\Z)^*$
are characters satisfying $\phi \psi=\chi_N$.
\item[(c)] $\overline{\rho}_{E,N}(G_K)$ is conjugate inside
$\GL_2(\Z/N\Z)$ to a subgroup of
\[
B_0(N) \; := \; 
\left\{ 
\begin{pmatrix}
a & b \\
0 & d
\end{pmatrix} \; : \; b \in \Z/N\Z, \quad a,~d \in (\Z/N\Z)^*\right\}
\subset \GL_2(\Z/N\Z).
\]
\end{enumerate}
\end{thm}
\begin{proof}
(a)$\implies$(b).
Suppose $E$ has a cyclic $K$-rational $N$-isogeny $\theta : E \rightarrow E$.
The $\ker(\theta)$ is cyclic of order $N$ and thus
$\ker(\theta)=\langle P \rangle$
where $P$ is an element of $E[N]$ of order $N$. As $\theta$ is
defined over $K$, the group $\langle P \rangle$ is $K$-rational
(i.e. it is stable under the action of $G_K$). Let $Q \in E[N]$
be such that $P$, $Q$ is a basis. Then for all $\sigma \in G_K$,
we have
\[
P^\sigma=a_\sigma P, \qquad Q^\sigma=b_\sigma P + d_\sigma Q.
\]
Hence
\[
\overline{\rho}_{E,N}(\sigma)=\begin{pmatrix}
a_\sigma & b_\sigma \\
0 & d_\sigma
\end{pmatrix}
\]
for all $\sigma \in G_K$.
Let $\phi$, $\psi : G_K \rightarrow (\Z/N\Z)^*$ be given by
$\phi(\sigma)=a_\sigma$, $\psi(\sigma)=d_\sigma$.
We leave to the reader the task of checking
that $\phi$, $\psi$ must be characters, and completing the remainder
of the proof.
\end{proof}

\section{Quadratic twisting}
\begin{lem}\label{lem:twistrep}
Let $d \in K^*$. Suppose
the characteristic of $K$ is not $2$. Let $E^\prime$ be the quadratic twist of $E$ by $d$.
Let $\psi  :  G_K \rightarrow \{1,-1\}$ be the quadratic
character defined by
$\sqrt{d}^\sigma=\psi(\sigma) \cdot \sqrt{d}$.
Then $\overline{\rho}_{E,N} \sim \psi \cdot \overline{\rho}_{E^\prime,N}$.
\end{lem}
\begin{proof}
As the characteristic is not $2$, the curves $E$, $E^\prime$ have models
\[
E \; : \; Y^2 \; = \; X^3+aX^2+bX+c, \qquad
E^\prime \; : \; Y^2 \; = \; X^3+daX^2+d^2bX+d^3c.
\]
The map
\[
\phi : E(\overline{K}) \rightarrow E^\prime(\overline{K}),
\qquad
\phi(x,y)\; =\; \left(\frac{x}{d}, \; \frac{y}{d\sqrt{d}}\right)
\]
is an isomorphism of abelian groups, and thus induces an isomorphism
$\phi  : E[N]\rightarrow E^\prime[N]$. Let $P=(x,y) \in E[N]$.
Note that $\pm P=(x,\pm y)$. Thus,
\[
\phi(P)^\sigma \; =\; 
\left(\frac{x^\sigma}{d}, \; \frac{y^\sigma}{d\sqrt{d}^\sigma}\right) \; = \;
\left(\frac{x^\sigma}{d}, \; \psi(\sigma) \cdot \frac{y^\sigma}{d\sqrt{d}}\right) \; = \;
\psi(\sigma) \cdot \left(\frac{x^\sigma}{d}, \; \frac{y^\sigma}{d\sqrt{d}}\right) \; = \;
\psi(\sigma) \cdot \phi(P^\sigma) .
\]
Now let $P$, $Q$ be a basis for $E[N]$, and we take
$\phi(P)$, $\phi(Q)$ as a basis for $E^\prime[N]$.
With respect to these bases it is now
an easy exercise to show that
$\overline{\rho}_{E,N}=\psi \cdot \overline{\rho}_{E^\prime,N}$.
\end{proof}

\begin{thm}
Let $H$ be a subgroup of $\GL_2(\Z/N\Z)$. Suppose that
$\overline{\rho}_{E,N}(G_K)$ is contained in $H$.
Let $E^\prime$ be a quadratic twist by some $d \in K^*$.
If $-I \in H$, then
$\overline{\rho}_{E^\prime,N}(G_K)$ is contained in a conjugate
of $H$.
\end{thm}
\begin{proof}
This follows immediately from Lemma~\ref{lem:twistrep}.
\end{proof}

\begin{cor}
If $E$ has a cyclic $K$-rational $N$ isogeny, then so
does any quadratic twist.
\end{cor}

For $N \ge 3$, if $E$ is an elliptic curve with a $K$-rational
point of order $N$
then a non-trivial quadratic twist will not have a point of order $N$,
but it will have an $N$-isogeny.

\begin{exercise}
Suppose $E$ has a $K$-rational $3$-isogeny. Show that there
is a quadratic twist $E^\prime$ that has a point of order $3$.
\end{exercise}

\section{Local properties of mod $N$ representations of elliptic curves}

Let $K$ be a number field and $\lambda$
be a prime of $K$. Let $E$ be an elliptic curve defined over $K$.
We say that $\overline{\rho}_{E,N}$ is \textbf{unramified} at $\lambda$
if $\overline{\rho}_{E,N}(I_\lambda)=1$, where $I_\lambda \subseteq G_K$ denotes an inertia
subgroup at $\lambda$.

\begin{thm}
Suppose $\lambda \nmid N$ is a prime of good reduction for $E$. Then
$\overline{\rho}_{E,N}$ is unramified at $\lambda$.
\end{thm}
\begin{proof}
The choice of inertia subgroup $I_\lambda$ corresponds to a choice of prime
$\mu$ of $\overline{K}$
above $\lambda$. As $E$ has good reduction at $\mu$ and $\mu \nmid N$, the reduction modulo $\mu$
map
\begin{equation}\label{eqn:redmodmu}
E[N] \rightarrow E(\overline{\F}_\lambda), \qquad Q \mapsto \tilde{Q} \pmod{\mu}
\end{equation}
is injective. Let $\sigma \in I_\lambda$. Then for all $Q \in E[N]$
we have that $\widetilde{Q^\sigma} = \tilde{Q}$ by definition of inertia. By the injectivity of \eqref{eqn:redmodmu}
we have $Q^\sigma=Q$. Thus $\overline{\rho}_{E,N}(\sigma)=1$ which completes the proof.
\end{proof}

\section{The mod $N$ representation of a Tate curve}
Another very instructive computation is the mod $N$
representation of a Tate curve.
The standard reference for Tate curves is Silverman's advanced
textbook \cite[Chapter V]{SilvermanII}
In this section $K$ is a field complete with respect
to a non-archimedean valuation $\lvert \cdot \rvert$
(e.g. $K=\Q_p$). Let $q \in K^*$ satisfy $\lvert q \rvert<1$.
Define
\[
s_k(q) \; :=\; 
\sum_{n \ge 1} \frac{n^k q^n}{1-q^n} \, , \qquad
a_4(q) \; := \; -s_3(q) \, , \qquad
a_5(q) \; := \; - \frac{5 s_3(q)+7 s_5(q)}{12}.
\]
These converge in $K$.
Define the \textbf{Tate curve} with parameter
$q$ by
\[
E_q \; : \; \qquad
Y^2+XY=X^3+a_4(q) X + a_6(q).
\]
This is an elliptic curve over $K$ with discriminant
\[
\Delta \; = \; q \prod_{n \ge 1} (1-q^n)^{24} \, ,
\]
and $j$-invariant
\begin{equation}\label{eqn:jq}
j \; = \; \frac{1}{q} + 744 + 196884 q^2 + \cdots .
\end{equation}
\begin{example}
If $E/\Q_p$ has split multiplicative reduction, then $E \cong E_q$
for some choice of $q \in \Q_p$ (i.e. $E$ is a Tate curve).
If $E/\Q_p$ has potentially multiplicative reduction (i.e. $\lvert j(E) \rvert_p >1$)
then $E$ is the quadratic twist of some Tate curve $E_q$ by $-c_4(E)/c_6(E)$
where $c_4$, $c_6$ has their usual meanings.
\end{example}

\begin{thm}[Tate]
There is an analytic isomorphism
\[
\phi \; : \; E_q(\overline{K}) \rightarrow \overline{K}^*/ q^{\Z} \, ,
\]
which is compatible with the action of $G_K$.
\end{thm}

\begin{cor}\label{cor:Tate}
Let $E=E_q$ be a Tate curve as above. Then
\[
\overline{\rho}_{E,N} \sim
\begin{pmatrix}
\chi_N & * \\
0 & 1 
\end{pmatrix} .
\]
\end{cor}
\begin{proof}
Note that $\phi$ induces an isomorphism
\[
\phi \; : \; E[N] \rightarrow \left(\overline{K}^*/q^\Z \right) [N]
\]
that is compatible with the action of $G_K$. A basis for the
group on the right is $\zeta_N$, $q^{1/N}$. Let $\sigma \in G_K$.
Then
\[
\sigma(\zeta_N) = \zeta_N^{\chi_N(\sigma)}, \qquad
\sigma(q^{1/N})= \zeta_N^{a_\sigma} q^{1/N}
\]
for some $a_\sigma$. Let $P=\phi^{-1}(\zeta_N)$, $Q=\phi^{-1}(q^{1/N})$.
Then $P$, $Q$ is a basis for $E[N]$ and, as $\phi$ is compatible
with the $G_K$-action
\[
P^\sigma= \chi_N(\sigma) \cdot P,
\qquad
Q^\sigma = a_\sigma \cdot P+ Q.
\]
Hence, with respect to this basis,
\[
\overline{\rho}_{E,N}(\sigma)=\begin{pmatrix}
\chi_N(\sigma) & a_\sigma \\
0 & 1
\end{pmatrix} \, .
\]
\end{proof}

\begin{example}
Let $E/K$ have split multiplicative reduction at $\fq$. Let $G_\fq \subset G_K$
be the decomposition group at $\fq$; this is simply $G_{K_\fq}$. As $E$
is a Tate curve when considered over $K_\fq$, we see from the above
that
\[
\overline{\rho}_{E,N} \vert_{G_\fq} \sim 
\begin{pmatrix}
\chi_N & * \\
0 & 1
\end{pmatrix}.
\]
More generally, let $E/K$ have potentially multiplicative reduction
at $\fq$. Let $\psi : G_\fq \rightarrow \{\pm 1\}$ be the character
satisfying $\sigma(\sqrt{-c_4/c_6})=\psi(\sigma) \cdot \sqrt{-c_4/c_6}$.
Then
\[
\overline{\rho}_{E,N} \vert_{G_\fq} \sim 
\psi \cdot \begin{pmatrix}
\chi_N & * \\
0 & 1
\end{pmatrix}.
\]
\end{example}

We shall later need the following lemma.
\begin{lem}\label{lem:transvection}
Let $E$ be an elliptic curve over a number field $K$,
and write $j$ for the $j$-invariant of $E$.
Let $\fq$ be a prime ideal of $\OO_K$.
Let $p$ be a prime.
Suppose
\[
v_\fq(j) <0, \qquad p \nmid \ord_\fq(j).
\]
Then $p \mid \# \overline{\rho}_{E,p}(I_\fq)$.
\end{lem}
\begin{proof}
Since $v_\fq(j)<0$, the elliptic curve $E$ 
has potentially multiplicative reduction at $\fq$.
Thus,
$E$ is a quadratic twist over $K_\fq$ of $E_q$ for some $q$. 
By \eqref{eqn:jq} we see that $p \nmid \ord_\fq(q)$.
Hence the extension $K_\fq(q^{1/p})/K_\fq$ is ramified.
Therefore, there is some $\sigma \in I_\fq$
such that $\sigma(q^{1/p})=\zeta_p^{a} q^{1/p}$, 
where $a \not \equiv 0 \pmod{p}$. Now, by the proof of 
Corollary~\ref{cor:Tate}, 
\[
\overline{\rho}_{E_q,p}(\sigma) = \begin{pmatrix}
\chi_p(\sigma) & a \\
0 & 1
\end{pmatrix}.
\]
As $E$ is a quadratic twist of $E_q$ over $K_\fq$,
we obtain
\[
\overline{\rho}_{E,p}(\sigma) = \pm \begin{pmatrix}
\chi_p(\sigma) & a \\
0 & 1
\end{pmatrix}.
\]
It is now an easy exercise, using the fact that $a \not\equiv 0 \pmod{p}$
the $p$ divides the order of this matrix.
Thus $p \mid \# \overline{\rho}_{E,p}(I_\fq)$.
\end{proof}
\section{Mod $p$ representations of elliptic curves}\label{sec:modp}
We now specialise to the case $N=p$, a prime. Then
$\overline{\rho}_{E,p}$ is called the mod $p$ representation
of $E$. We say that $\overline{\rho}_{E,p}$ is \textbf{reducible}
if there some $P \in E[p]$, $P \ne 0$, such that
$\sigma(P)=a_\sigma P$ for all $\sigma \in G_K$.
Note that $P$ does not have to be fixed by $G_K$, but
the cyclic subgroup of $\langle P \rangle$ is fixed by $G_K$.
Thus, by Theorem~\ref{thm:pisog}, the mod $p$ representation
$\overline{\rho}_{E,p}$ is reducible if and only if $E$
has a $p$-isogeny defined over $K$. Equivalently,
\[
\overline{\rho}_{E,p} \thicksim \begin{pmatrix} * & * \\ 0 & * \end{pmatrix}.
\]
We note that if $\overline{\rho}_{E,p}$ is reducible,
then its image is (up to conjugation) contained in the Borel
subgroup 
\begin{equation}\label{eqn:Borel}
B_0(p):=\left\{
\begin{pmatrix} a & b\\
0 & d 
\end{pmatrix} \in \GL_2(\F_p) \; : \; 
a,b,d \in \F_p
\right\}
\end{equation}
 of $\GL_2(\F_p)$.
If $\overline{\rho}_{E,p}$ is not reducible then we say it is
\textbf{irreducible}.

It is often useful to know Dickson's classification of
subgroups of $\GL_2(\F_p)$.
\begin{thm}[Dickson]
Let $H$ be a subgroup of $\GL_2(\F_p)$ not containing
$\SL_2(\F_p)$. Then (up to conjugation)
\begin{enumerate}
\item[(i)] either $\displaystyle H \subseteq B_0(p):=\left\{
\begin{pmatrix} * & *\\
0 & * 
\end{pmatrix}
\right\}$ (Borel subgroup)
\item[(ii)] or $\displaystyle H \subseteq N_s^+(p):=\left\{
\begin{pmatrix} \alpha & 0\\
0 & \beta 
\end{pmatrix} , \;
\begin{pmatrix} 0 & \alpha\\
\beta & 0
\end{pmatrix} \; : \; \alpha,~\beta \in \F_p^*
\right\}$ (normalizer of split Cartan)
\item[(iii)] or
$\displaystyle H \subseteq N_{ns}^+(p)$ (normalizer
of non-split Cartan)\footnote{
$N_{ns}^+(p)$ can be conjugated inside $\GL_2(\F_{p^2})$
to
\[
\left\{
\begin{pmatrix} \alpha & 0\\
0 & \alpha^p
\end{pmatrix} , \;
\begin{pmatrix} 0 & \alpha\\
\alpha^p & 0
\end{pmatrix} \; : \; \alpha\in {\F_{p^2}}^\ast
\right\}. 
\]
}
\item[(iv)] or the image of $H$ in $\PGL_2(\F_p)$
is isomorphic to $A_4$, $S_4$ or $A_5$ (these are
called the exceptional subgroups of $\GL_2(\F_p)$).
\end{enumerate}
\end{thm}

\begin{exercise}
Suppose $E$ has potentially multiplicative reduction at $\fq$.
Show for $p$ sufficiently large
that $\overline{\rho}_{E,p}(G_\fq)$ is not exceptional.
(Hint: use Lemma~\ref{lem:transvection}).
\end{exercise}

\section{Irreducibility of mod $p$ representations of elliptic curves}
We start with what is believed
to be correct. 
The following conjecture is widely known as Serre's uniformity 
conjecture over number fields (e.g. \cite[Conjecture 1.3]{BELOV}).
\begin{conj}
Let $K$ be a number field. Then there is
some constant $C_K$ such that the following
holds.
Let $E/K$
be an elliptic curve without CM,
and let $p>C_K$ be a prime. 
Then $\overline{\rho}_{E,p}$ is surjective.
\end{conj}
Note that if $\overline{\rho}_{E,p}$ is surjective then
it certainly irreducible. However, even for $K=\Q$,
Serre's uniformity conjecture is still unresolved.
For recent progress see \cite{BiluParent}, \cite{FournLemos}, \cite{Lemos}.


Mazur's isogeny theorem \cite{Mazur78} asserts that if $E$
is an elliptic curve over $\Q$, and $p>163$ is a prime,
then $E$ does not have $p$-isogenies defined over $\Q$.
This immediately implies that $\overline{\rho}_{E,p}$
is irreducible. Unfortunately, this is no analogue for 
Mazur's isogeny theorem over number fields. However,
in order to apply the level-lowering theorem
(Theorem~\ref{thm:levellowering}) we need to show that
$\overline{\rho}_{E,p}$ is irreducible
for elliptic curves $E$ defined over a totally real field $K$.

We recall the following result of 
Kraus \cite[Corollary 1]{Kraus96} 
concerning semistable elliptic curves 
over quadratic fields.

\begin{thm}[Kraus]
	\label{Kraus:quadratic}
Let $K$ be a quadratic 
field of class number $1$.
Let $p\geq 5$ be a prime.
Suppose $E$ is semistable.
Then if $\modpg$ is reducible 
then $p\leq 13$ or $p\mid D_{K}\cdot N_{K}(\varepsilon^2-1)$, 
where $D_{K}$ is the discriminant of $K$, $N_{K}$ denotes the 
absolute norm of $K$, and $\varepsilon$ is the fundamental unit of $K$.
\end{thm}
\begin{example}\label{ex:Kraus}
The application of Theorem \ref{Kraus:quadratic} 
was a key step in the resolution of the Fermat 
equation over $K=\Q(\sqrt{2})$ by Jarvis and Meekin \cite{Frazer04}.
In that case, the Frey curve $E$ is semistable. A fundamental
unit for $K$ is $\varepsilon=1+\sqrt{2}$.
Then $D_K=8$ and $N_K(\varepsilon^2-1)=-4$.
Thus for $p \ge 17$, Theorem \ref{Kraus:quadratic}
tells us the mod $p$ representation $\overline{\rho}_{E,p}$
is irreducible. We consider irreducibility
for smaller values of $p$ later.
\end{example}
Kraus \cite[Theorem 1]{Kraus07} generalized
Theorem~\ref{Kraus:quadratic}
to more general numbers fields.
Later work of Freitas and Siksek \cite{Freitas15} 
asserts a more general result by 
building upon work of David \cite{David12} 
and Merel's uniform boundedness theorem \cite{Merel96}.
\begin{thm}[Merel] \label{thm:Merel}
Let $K$ be a number field of degree $d$. 
Let $E$ be a an elliptic curve over $K$. 
If $E$ has a point of prime order $p$ defined 
over $K$ then 
\[
		p<(1+3^{\frac{d}{2}})^2.
\]	
\end{thm}
\begin{thm}[Freitas and Siksek]
Let $K$ be a totally real field. 
There is an effectively computable
constant $C_K$, depending only on $K$
such that the following holds.
Let $E$
be an elliptic curve over $K$.
Let $p>C_K$ be a prime, and suppose
$E$ is semistable at all the primes of $K$ above $p$.
Then $\modpg$ is irreducible. 
\end{thm}

\section{Relationship to $X_{0}(N)$}
It is often necessary to prove the irreducibility of $\modpg$ for small primes $p$ 
separately. 
In these cases it is convenient to exploit the relationship to certain modular
curves.  We outline this relationship below.

Write 
\[
\mathbb{H}=\{x+iy\; : \; x,\; y\in \R, \; y >0\}
\]
for the \textbf{upper-half plane} 
and let $\mathbb{H}^{\ast}=\mathbb{H}\cup \Q \cup \{\infty\} $ denote 
the \textbf{extended upper-half plane}. Recall that $\SL_2(\Z)$
acts on $\mathbb{H}$ via fractional linear transformations
\[
\begin{pmatrix}
a & b\\
c & d
\end{pmatrix} \cdot \tau =\frac{a \tau+b}{c \tau+d}.
\]
For an integer $N\geq 1$, let
\[
\Gamma_{0}(N)=\left\{
\begin{pmatrix}
a & b\\
c & d
\end{pmatrix}
\;
\in \; \SL_{2}(\Z) \;:\; c\equiv 0  \pmod{N} \right\}.
\]
Write $\Gamma_0(N) \backslash \mathbb{H}$ for the quotient
of $\mathbb{H}$ by the group $\Gamma_0(N)$. This quotient
is a non-compact Riemann surface, which turns out
to be isomorphic to the set of complex points $Y_0(N)(\mathbb{C})$
on a (non-compact) algebraic curve $Y_0(N)$ defined over $\Q$.
Let $E_{1}$, $E_{2}$ be elliptic curves defined over $\mathbb{C}$
and let $C_{1}$, $C_{2}$ be cyclic subgroups of order $N$ on $E_{1}$ and $E_{2}$, 
respectively. 
We say the pairs $(E_{1}, C_{1})$ and $(E_{2}, C_{2})$ 
are \textbf{isomorphic} if there is an isomorphism $\phi: E_{1}\rightarrow E_{2}$ 
such that $\phi(C_{1})=C_{2}$.
There is a one-to-one correspondence
\[
Y_0(N)(\mathbb{C}) \cong \Gamma_{0}(N) \backslash \mathbb{H}\leftrightarrow 
\{\text{isomorphism classes of pairs}\; (E/\mathbb{C},\; C)  \}
\]
where $E$ is an elliptic curve over $\mathbb{C}$
and $C$ is a cyclic subgroup of order $N$ on $E$. 
We let $X_0(N)$ be the compactification of $Y_0(N)$.
Then,
\[
X_{0}(N)(\mathbb{C})\cong \Gamma_{0}(N) \backslash\mathbb{H}^{\ast}.
\]
We call $X_{0}(N)(\mathbb{C})\setminus Y_{0}(N)(\mathbb{C})$ the \textbf{set of cusps} of $X_{0}(N)$. 
For a more detailed construction of the modular curve $X_{0}(N)$ we refer 
the reader to \cite{Diamond} or \cite{modcurves}.

\medskip

Let $p$ be a prime, $E/K$ an elliptic curve,
and suppose $\modpg$ is reducible. As explained
in Section~\ref{sec:modp}, there is a subgroup $C=\langle P \rangle$
of $E(\overline{K})$ of order $p$ 
that is $G_{K}$-stable. 
This gives us a non-cuspidal point $(E,C)$ on $X_0(p)$. In fact,
as $E$ and $C$ are both fixed by $G_K$,
the non-cuspidal point we obtain belongs to $X_0(p)(K)$.
In particular, to demonstrate the irreducibility of $\modpg$, 
it suffices to show that $X_{0}(p)(K)$ consists only of cuspidal points. 
Moreover, if $E$ has a point of order $2$ defined over $K$ (or full $2$-torsion over $K$) 
then it suffices to 
show that $X_{0}(2p)(K)$ (or $X_{0}(4p)(K)$) 
respectively consist only of cuspidal points.

\begin{example}
Let us return to the Frey curve $E$ associated to
the Fermat equation over $K=\Q(\sqrt{2})$.
In Example~\ref{ex:Kraus} we used Kraus' theorem
to show that $\overline{\rho}_{E,p}$ 
is irreducible for $p \ge 17$. 
Let us consider $p=5$. Note that the Frey curve $E$
has full $2$-torsion. Thus, if $\overline{\rho}_{E,5}$ 
is reducible then we obtain a non-cuspidal
$K$-point on $X_0(20)$. The \texttt{Small modular curves package}
in \texttt{Magma} allows us to write down an equation for $X_0(20)$:
\[
X_0(20) \; : \; y^2 = x^3 + x^2 + 4 x + 4.
\]
This is an elliptic curve. We're interested in computing
the Mordell--Weil group $X_0(20)(K)$. Using
\texttt{Magma} (which uses standard descent
algorithms) we find that
\[
X_0(20)(K)=X_0(20)(\Q)=\{\OO,(4,\pm 10),~(0,\pm 2),~(-1,0)\}.
\]
We can also check, using
the \texttt{Small modular curves package},
that these six points are cusps. Thus
$\overline{\rho}_{E,5}$ is irreducible. For a glimpse
at the ideas behind the algorithms used in this
package, see the survey \cite{modcurves}.
\end{example}
For other examples of the determination
of all points on $X_0(N)$ over a fixed
number field, see
 \cite{PhilippeBielliptic} 
or \cite[Section 4]{Fermat23}.

\section{Non-cuspidal points on $X_{0}(N)$}
Mazur \cite{Mazur77} famously showed that 
$X_{0}(p)(\Q)$ only consists of cuspidal points for all 
but finitely many primes $p$. 
More precisely he proved that $X_{0}(p)(\Q)$ possesses 
a non-cuspidal point if and only if
\[
	p\in \{2, 3, 5, 7, 11, 13, 17, 19, 37, 43, 67, 163\}.
\] 


\medskip

In a series of papers culminating in \cite{Kenku82}, Kenku 
completed the classification begun by Mazur by showing that $X_{0}(N)(\Q)$
possesses 
a non-cuspidal point precisely when
\[
N\in \{1-19, 21, 25, 27, 37, 43, 67, 163\}.
\]
Although there are currently no analogous results 
for higher degree number fields, there has been a significant 
advancement in the understanding of all 
quadratic points on $X_{0}(N)$ for small $N$ 
(see e.g. 
\cite{Adzaga}, 
\cite{banwait2022cyclic},
\cite{Box},
\cite{BruinNajman},
\cite{Philippe}, 
\cite{Najman}, 
\cite{OzmanSiksek}).
The general understanding of higher degree points 
on these modular curves is still rudimentary. 
There is however some respectable progress,
and we refer the reader to \cite{Banwait22}, \cite{Box23} and 
\cite{Primitive} for results established in this direction.  

\medskip
It is natural to ponder whether a modular curve $X_{0}(N)$ can 
have infinitely many points of a fixed degree.
We begin our exposition with quadratic points. 
Recall that a curve $C/\Q$ is \textbf{hyperelliptic} 
if $C$ admits a degree $2$ map to $\PP^1$ 
defined over $\Q$. We can write
\[
	C: Y^2=f(X), \qquad f\in\Q[X],\quad \deg(f)\geq 5.
\]
Note that there are only finitely many $\alpha\in\Q$ 
with $\sqrt{f(\alpha)}\in\Q$ since by $C(\Q)$ 
is finite by Faltings' theorem \cite{Faltings}. 
For all other $\alpha \in \Q$,
we obtain a quadratic point $(\alpha,\sqrt{f(\alpha)})$.
Recall that a curve $C/\Q$ 
is \textbf{bielliptic} if it has genus $\ge 2$ and
$C$ admits a degree $2$ map to 
an elliptic curve $E$ defined over $\Q$. 
Let 
\[
	\pi:C\rightarrow E
\]
denote the corresponding degree $2$ map.
Suppose $E(\Q)$ is infinite (i.e. $E$ has positive rank
over $\Q$). 
Then, for all but finitely $P \in E(\Q)$,
the two points in the fibre
 $\pi^{-1}(P)$ 
are quadratic points on $C$.
In fact a theorem of Harris and Silverman
\cite[Corollary 3]{HS} asserts that if a curve
$C/\Q$ of genus $\ge 2$ has infinitely many quadratic 
points then $C$ is either hyperelliptic or bielliptic. 
Ogg \cite{Ogg} has determined all $N$ for which $X_{0}(N)$ 
is hyperelliptic, 
and Bars \cite{Bars} has determined all $N$ for which 
$X_{0}(N)$ is bielliptic. 
Furthermore Jeon \cite{Jeon21} has found all such $N$ such that $X_{0}(N)$ 
has infinitely many cubic points, 
and Hwang and Jeon \cite{HwangJeon} and 
Derickx and Orli\'{c} \cite{DerickxOrlic} 
have independently determined all $N$ for which
$X_{0}(N)$ has infinitely many quartic points. 

\medskip

The modular curve $X=X_{0}(37)$ is 
an outlier in the sense that it has two 
infinite sources of quadratic points; 
it is both hyperelliptic and bielliptic. 
Box \cite[Proposition 5.4]{Box} 
has given a description of all 
quadratic points on $X_{0}(37)$.
In his study of the Fermat equation 
over real quadratic fields, 
Michaud-Jacobs \cite[Lemma 4.9]{PhilippeFermat} 
proved the following result which was inspired 
by Kamienny's formal immersion criterion 
\cite[pg. 223-225]{Kamienny}.

\begin{lem}[Michaud-Jacobs]
Let $K$ be a quadratic field. 
Let $E$ be an elliptic curve defined over $K$. 
Suppose $E$ has a point of order $2$ defined over $K$.
Suppose $E$ has multiplicative reduction 
at all primes of $K$ above $q$, where $q\neq 2, 37$ is a prime. 
Then $\overline{\rho}_{E,37}$ is irreducible.
\end{lem}

Michaud-Jacobs \cite{PhilippeBielliptic}
later developed a method to compute 
all points on bielliptic $X_{0}(N)$ defined over 
a fixed quadratic field; 
the values $N=37$ and $N=43$ 
are not amenable to this method however.

\section{Relationship to $X_{1}(N)$}
Let $K$ be a number field and 
let $E$ be an elliptic curve 
over $K$.
Several arguments in the literature show
that one possible consequence of $\modpg$ 
being reducible is that $E$ has a point of order $p$ 
defined over an extension of $K$ (see e.g. \cite[Lemma 10.2]{Anni}, \cite[Lemma 6.1]{FLTsmall}, 
\cite[Lemma 3.3]{FS15}, \cite[Lemma 4.4]{Fermat23}).
Recall that non-cuspidal $K$-points 
on the modular curve $X_{1}(p)$ correspond 
to pairs $(E,P)$ where $E$ is an elliptic 
curve defined over $K$ and $P$ is a point 
of order $p$ defined over $K$ 
(see e.g. \cite{Diamond} or \cite{modcurves}).
Merel's uniform boundedness theorem (Theorem \ref{thm:Merel}) asserts that for prime $p$,
and for $d$ satisfying $(3^{d/2}+1)^2 \le p$, the only
degree $d$ points on $X_1(p)$ are cuspidal. 
There is a smaller bound on $p$ for number fields 
of low degree. 
We recall that $p$ is a \textbf{torsion
prime of degree} $d$ if $X_{1}(p)(K)$ 
contains a non-cuspidal point 
for a number field $K$ of degree $d$.
Mazur \cite{Mazur77} showed that if
$p$ is a torsion prime of degree $1$ then $p \le 7$.
We now know the torsion primes of
degree $d$, for each $1\leq d\leq 8$, since Mazur's
work has since been extended by Kamienny \cite{Kamienny}, Parent \cite{Parent, Parent17},
Derickx, Kamienny, Stein and Stoll \cite{DKSS}, and Khawaja \cite{Khawaja}.
We summarise these results in Table \ref{tab:torsionprimes}. 

\begin{table}
\begin{center}
\begin{tabular}{ |c|c|c| } 
 \hline
 $d$ & Torsion primes of degree $d$ & References\\
 \hline
 $1$ & $\{p\leq 7 : p \text{ is prime}\}$ & Mazur \cite{Mazur77} \\ 
 \hline
$2$ & $\{p\leq 13 : p \text{ is prime}\}$ & Kamienny \cite{Kamienny}  \\ 
 \hline
 $3$ & $\{p\leq 13 : p \text{ is prime}\}$ & Parent \cite{Parent, Parent17}\\
 \hline
 $4$ & $\{p\leq 17 : p \text{ is prime}\}$ & Derickx, Kamienny, Stein and Stoll \cite{DKSS}\\
 \hline
  $5$ & $\{p\leq 19 : p \text{ is prime}\}$ & Derickx, Kamienny, Stein and Stoll \cite{DKSS}\\
   \hline
  $6$ & $\{p\leq 19 : p \text{ is prime}\}\cup\{37\}$ & Derickx, Kamienny, Stein and Stoll \cite{DKSS}\\
     \hline
  $7$ & $\{p\leq 23 : p \text{ is prime}\}$ & Derickx, Kamienny, Stein and Stoll \cite{DKSS}\\
     \hline
  $8$ & $\{p\leq 23 : p \text{ is prime}\}$ & Khawaja \cite{Khawaja}\\
  \hline
\end{tabular}
	\caption{Recall that $p$ is a torsion 
	prime of degree $d$ if the modular curve $X_{1}(p)$ contains a 
	non-cuspidal point defined over a number field of degree $d$. 
	This table summarises the exact determination 
	of the set of torsion primes of low degree.}
\end{center}
	\label{tab:torsionprimes}
\end{table}

\part{Modularity of elliptic curves}\label{sec:modularity} 

\section{Hilbert modular forms}
We refer the reader to 
\cite[Section 2]{DembeleVoight} 
for definitions. We content
ourselves with treating Hilbert modular
forms as black-boxes without actually defining them. 
Let $K$ be a totally real number field of degree $d$. 
Write $\OO_{K}$ for the ring of
integers of $K$.
Let $\mathbf{k}=(k_1,k_2,\dotsc,k_d)$ be a list of $d$
positive
integers, and $\mathcal{N}$ be an ideal of $\OO_K$.
Attached to this data is a space
$S_{\mathbf{k}}(\mathcal{N})$ of Hilbert cusp forms
of weight $\mathbf{k}$ and level $\mathcal{N}$.
On this space there is a natural action
of Hecke operators, which leads to the notion
of Hecke eigenforms. There is also a notion
of newforms of weight $\mathbf{k}$ and level
$\mathcal{N}$. These are simultaneous
eigenvectors to all the Hecke eigenforms,
that do not arise from smaller levels $\mathcal{N}^\prime
\mid \mathcal{N}$.
There are finitely many Hilbert newforms $\ff$
of any weight and level, and there are
effective algorithms for computing these,
including the computation of the Hecke eigenfield $\Q_\ff$
and the Hecke eigenvalues. A survey of these algorithms
can be found in \cite{DembeleVoight}.
We point out that there is a 
\texttt{Hilbert modular forms} package in the computer
algebra system
\texttt{Magma} \cite{Magma}, as well as a database 
for Hilbert modular forms on the \texttt{LMFDB} \cite{LMFDB}. 

\section{Modularity of elliptic curves over totally real fields} 
Let $E$ be an elliptic curve defined over a 
totally real number field $K$,
and write $\mathcal{N}$ for the conductor of $E$. 
We say $E$ is \textbf{modular} if there
is a Hilbert newform $\ff$ over $K$,
of parallel weight $2$ (i.e. weight $(2,2,\dotsc,2)$,
and level $\mathcal{N}$, having rational
Hecke eigenvalues,
such that $L(E,s)=L(\ff,s)$ or equivalently
$\rho_{E,p}\sim \rho_{\ff,p}$.
If $K=\Q$, then a Hilbert newform
is the same as a classical newform,
and modularity of elliptic curves is already known
thanks to Wiles, Breuil, Conrad, Diamond and Taylor (Theorem~\ref{thm:modularity}). We give a brief survey of modularity results
for elliptic curves over totally real fields, and 
refer the reader to Thorne's survey \cite{Thorne23} for 
a more comprehensive overview.


Jarvis and Manoharmayum \cite{JM08} were the first to 
establish modularity of infinite
families of elliptic curves over a number field
of degree $>1$. 
\begin{thm}[Jarvis and Manoharmayum]
Let $E$ be a semistable elliptic curve
over 
$\Q(\sqrt{2})$ or $\Q(\sqrt{17})$.
Then $E$ is modular.
\end{thm}

\section{Modularity lifting and switching} 
The following theorem 
builds on modularity lifting
theorems due to many authors,
notably Wiles, Taylor, Skinner, Kisin, 
Barnet-Lamb, Gee, Geraghty, Breuil, Diamond,
and modularity switching arguments
due to Wiles \cite{wilesswitch}  
and Manoharmayum \cite{Man01, Man04}.
The theorem appears in the paper of
Freitas, Le Hung and Siksek \cite[Theorem 3 \& 4]{quadmod},
but is in fact a synthesis of many earlier results.

\begin{thm}[Wiles, Manoharmayum, and Freitas, Le Hung and Siksek]
	\label{thm:357}
Let $E$ be an elliptic curve over a totally real field $K$. 
Let $p=3$, $5$ or $7$. 
If 
\[
(\overline{\rho}_{E, p}(G_K)) \cap \SL_2(\F_p)
\]
 is absolutely 
irreducible then $E$ is modular. 
\end{thm}

Let's explain what absolute irreducibility here means.
Let $H$ be a subgroup of $\GL_2(\F_p)$. We say that
$H$ is \textbf{reducible} if there is an  eigenvector
$\vv \in \F_p^2 \setminus \{ \mathbf{0}\}$ common to all
the matrices in $H$. If $H=\overline{\rho}_{E,p}(G_K)$,
then $\overline{\rho}_{E,p}$ is reducible if and only if $H$
is reducible. To see this, note that if $\overline{\rho}_{E,p}$
is reducible then its image $H$ is contained in the Borel
subgroup \eqref{eqn:Borel}, and so 
\[
\vv=\begin{pmatrix} 
1 \\
0
\end{pmatrix}
\]
is a common eigenvector for all matrices in $H$.
Conversely, if all matrices have a common eigenvector $\vv \in \F_p^2\setminus \{\mathbf{0}\}$, then we can conjugate $H$ so that it is contained
in the  Borel subgroup.

We say that a subgroup $H$ of $\GL_2(\F_p)$ 
is \textbf{absolutely reducible} if there an eigenvector
$\vv \in \overline{\F}_p^2 \setminus \{\mathbf{0}\}$
common to all the matrices in $H$.
\begin{example}
Let $H$ be a cyclic subgroup of $\GL_2(\F_p)$.
Then $H$ is absolutely reducible. Indeed,
let $A$ be a generator of $H$. Then $A$
has an eigenvector $\vv \in \overline{\F}_p^2 \setminus \{\mathbf{0}\}$.
Now $\vv$ is also an eigenvector for all powers $A^k$,
and so $H$ is absolutely reducible.

For a concrete example, let's return to the setting of Section~\ref{sec:example}. Here $p=2$ and $E$ is an elliptic curve of the form
$y^2=f(x)$ over a field of characteristic $\ne 2$. We consider the
case where $f$ is irreducible but $\Delta(f) \in (K^\ast)^2$.
We found that $H=\overline{\rho}_{E,2}(G_K)$ is cyclic of order $3$,
generated by the matrix
\[
A=\begin{pmatrix}
1 & 1\\
1 & 0
\end{pmatrix}.
\]
We note that $A$ has characteristic polynomial
$X^2-X-1$, which is irreducible over $\F_2$.
Therefore $H$ is irreducible (i.e. $\overline{\rho}_{E,2}$
is irreducible). However, the characteristic polynomial
has roots in $\F_4$,
hence $A$ has an eigenvector
in $\F_4^2 \setminus \{\mathbf{0}\}$, and so $H$
is absolutely reducible.
\end{example}

\section{A general modularity theorem for elliptic curves over totally real fields}
As a consequence of Theorem \ref{thm:357}, a 
non-modular elliptic curve over a totally real field $K$ 
corresponds to a non-cuspidal $K$-point on one of seven
complicated  modular curves 
(see \cite[Section 11]{quadmod}). 
These curves have genera $3$, $3$, $4$, $73$, $97$, $113$ and $153$.
By Faltings' theorem, there are at most finitely many $K$-points
on each of these curves. This 
simple observation yields  
the following theorem \cite[Theorem 5]{quadmod}.
\begin{thm}[Freitas, Le Hung and Siksek]\label{thm:exceptions}
Let $K$ be a totally real field. Then there
are at most only finitely many $j$-invariants
of non-modular elliptic curves defined over $K$.
\end{thm}

\begin{example}
Let us give an example of how Theorem~\ref{thm:exceptions}
is applied to Frey curves. 
Let $K$ be a totally real field, and let $j_1,j_2,\dotsc,j_s$
be the finite list of non-modular $j$-invariants whose
existence is asserted by Theorem~\ref{thm:exceptions}.
Consider the Frey curve
\[
E \; : \; Y^2=X(X-a^p)(X+b^p)
\]
attached to a non-trivial integral solution $(a,b,c)$ of the Fermat
equation 
\[
a^p+b^p+c^p=0,
\]
Write $u=a^p$, $v=b^p$, $w=c^p$.
Then, the $j$-invariant of $E$ is
\[
j=256\frac{(w^2-uv)^3}{u^2 v^2 w^2}.
\]
Suppose $E$ is non-modular. Then $j=j_i$ for some $i=1,\dotsc,s$. Thus 
the triple $(u,v,w)$ satisfies the two relations
\[
u+v+w=0, \qquad 256 (w^2-uv)^3= j_i \cdot u^2 v^2 w^2.
\]
Eliminating $w$ we have
\[
256 (u^2+uv+v^2)^3 - j_i u^2v^2(u+v)^2=0.
\]
This is a homogeneous equation, and dividing by $v^6$
gives us that $u/v$ is a root of a polynomial
equation of degree at most $6$. Thus, we obtain
a finite set $\alpha_1,\dotsc,\alpha_n \in K$
of possibilities
for the ratio $(a/c)^p=u/v$. But, for $p$ sufficiently large,
the equation $(a/c)^p=\alpha_k$ has no solutions unless
$\alpha_k$ is a root of unity. As $K$ is real, the only roots
of unit in $K$ are $\pm 1$. Thus, for $p$ sufficiently large
$a=\pm c$ and similarly $b=\pm c$, which contradicts
$a^p+b^p+c^p=0$ and $abc \ne 0$.
\end{example}

\section{Modularity of elliptic curves over low degree fields}
As discussed previously, thanks to Theorem~\ref{thm:357},
elliptic curves over a totally real field $K$
that fail to be modular give rise to $K$-points
on one of seven complicated modular curves.
Methods
for computing low degree points on modular curves
have been applied to these complicated modular curves,
yielding the following theorem \cite{quadmod}.
\begin{thm}[Freitas, Le Hung and Siksek]
Elliptic curves over real quadratic fields
are modular.
\end{thm}
Later work of Thorne \cite{Thorne} and Kalyanswamy \cite{Kalyanswamy} 
refined Theorem \ref{thm:357}, and resultingly 
a non-modular elliptic curve defined over a totally real field $K$ 
satisfying $\sqrt{5} \notin K$, and $\Q(\zeta_7) \cap K=\Q$,
corresponds to a non-cuspidal $K$-point on one of four
modular curves (see \cite[Section 1.2]{quarticmod}). 
These improvements by Thorne and Kalyanswamy
paved the way for the following two
theorems due to Derickx, Najman and Siksek \cite{cubmod}
and to Box \cite{quarticmod}.
\begin{thm}[Derickx, Najman and Siksek]
Elliptic curves over totally real cubic
fields are modular.
\end{thm}

\begin{thm}[Box]
Let $K$ be a totally real quartic field
and suppose $\sqrt{5} \notin K$.
Let $E$ be an elliptic curve defined
over $K$. Then $E$ is modular.
\end{thm}
For the proof of his theorem, Box uses 
the newly developed technique of relative
symmetric Chabauty, due
to Box, Gajovi\'{c} and Goodman \cite{Box23},
to control quartic points on certain quotients
of these modular curves.

\section{Higher degree number fields} 
It is natural to ask whether it is currently plausible to establish the 
modularity of all elliptic curves over all
 totally real fields of a higher fixed degree. 
Box \cite[Section 7.2]{quarticmod} ponders this 
question and observes, for example, that one current limitation of establishing the modularity of 
elliptic curves over totally real quintic fields is that the Chabauty 
bound (which holds for degree $4$ points) fails to hold for degree $5$ points. 
A (non-effective) result of Ishitsuka, Ito and Yoshikawa \cite[Theorem 1.2]{IIY22}
asserts that elliptic curves over all but finitely many totally 
real quintic fields are modular.

We highlight the existence of results that assert the modularity of  
elliptic curves over higher degree number fields subject to certain assumptions 
on the number field or the elliptic curve e.g. \cite[Theorem 4.1]{DF13}, \cite[Theorem 6.2]{FreitasFrey}, \cite{Anni}, \cite{Yoshikawa}, \cite{Zhang}, 
\cite[Theorem 7]{quadmod}. 
The following result
of Freitas and Siksek \cite[Theorem 7]{quadmod} 
shows that is now possible to deduce modularity results
for elliptic curves using relatively elementary
computations.
\begin{thm}[Freitas and Siksek]
	\label{thm:somemodular}
Let $K$ be a totally real field.
Let $p=5$ or $p=7$. Let $\fp$
be a prime ideal of $\OO_K$ above $p$,
and suppose $K$ is unramified at $\fp$.
Let $E$ be an elliptic curve over $K$ that is
 semistable at $\fp$ 
and suppose $\overline{\rho}_{E,p}$ is irreducible.
Then $E$ is modular.
\end{thm} 
\begin{proof}
We sketch the idea of the proof.
The fact that $E$ is semistable at $\fp$ (which in turn
is unramified) forces the image of  
$\overline{\rho}_{E,p}(I_\fp)$ in $\PGL_2(\F_p)$
to contain a cyclic subgroup of order $p-1$ or $p+1$.
This, together with the assumption that $\overline{\rho}_{E,p}$
is irreducible, is enough to imply that $\overline{\rho}_{E,p}(G_K) \cap \SL_2(\F_p)$ is absolutely irreducible. Thus, by Theorem~\ref{thm:357},
$E$ is modular.
\end{proof}

\begin{example}
Kraus \cite[Lemma 9]{Kraus19} applied
Theorem~\ref{thm:somemodular} to 
deduce the modularity of the Frey curve (associated 
to a putative solution to the Fermat equation) over 
degree $5$ number field
$K=\Q(\alpha)$ 
where $\alpha^5-6\alpha^3+6\alpha-2=0$.
\end{example}


\part{Newform elimination}\label{sec:Elimination}
The final step in the modular approach is 
to eliminate the Hilbert newforms that arise 
from level-lowering, i.e. to show 
that the isomorphism
\begin{equation}\label{eqn:isom}
\modpg\sim \overline{\rho}_{\mathfrak{f}, \varpi}
\end{equation}
doesn't hold. 
\section{A bound for the exponent}
Recall that in our setting the prime $p$
in relation \eqref{eqn:isom}
is usually the prime exponent in some Diophantine
equation, and the elliptic curve $E$ is the Frey 
curve attached to that Diophantine equation.
In this section we explain a standard method
that is often capable of giving a bound on $p$.
The main idea behind this method essentially goes 
back to Serre \cite{Serre87}.
\begin{lem}\label{lem:newformelim}
Let $K$ be a totally real field, and let $p \ge 5$. 
Let $E$ be an elliptic curve over $K$
of conductor $\cN$,
and let $\ff$ be a newform of parallel weight $2$
and level $\cN_p$ 
given by \eqref{eqn:level}.
Let $t$ be a positive integer satisfying $t \mid \# E(K)_{\mathrm{tors}}$.
Let $\fq\nmid t \mathcal{N}_{p}$ be a prime ideal of $\OO_K$ and let 
\[
\mathcal{A_{\mathfrak{q}}}=\{a\in\Z:\quad \lvert a \rvert\leq 2\sqrt{\Norm(\mathfrak{q})}\;, 
\quad \Norm(\mathfrak{q})+1-a\equiv 0 \pmod{t} \}.
\]
If $\modpg\sim \overline{\rho}_{\mathfrak{f},\varpi}$ 
then $\varpi$ divides the principal ideal 
\[
B_{\mathfrak{f},\mathfrak{q}}=\Norm(\mathfrak{q})(
\Norm(\frak{q})+1)^2-a_{\mathfrak{q}}(\mathfrak{f})^2)\prod\limits_{a\in\mathcal{A_{\mathfrak{q}}}}(a-a_{\mathfrak{q}}(\mathfrak{f}))\cdot \mathcal{O}_{\Q_{\mathfrak{f}}}.
\]
\end{lem}
\begin{proof}
Suppose $\overline{\rho}_{E,p} \sim \overline{\rho}_{f,\varpi}$.
Let $\fq \nmid  t\cN_p$. If $\fq \mid p$ then
$\varpi \mid p \mid \Norm(\fq) \mid B_{\ff,\fq}$.
So suppose $\fq \nmid p$.
Let $\sigma_\fq \in G_K$ denote a Frobenius element
corresponding to $\fq$. Then $\overline{\rho}_{E,p}(\sigma_\fq)$
and $\overline{\rho}_{f,\varpi}(\sigma_\fq)$ are similar matrices
in $\GL_2(\F_\varpi)$, where $\F_\varpi=\OO_\ff/\varpi$, and therefore
\[
\Trace(\overline{\rho}_{E,p}(\sigma_\fq)) \equiv 
\Trace(\overline{\rho}_{\ff,\varpi}(\sigma_\fq)) 
\equiv a_\fq(\ff) \pmod{\varpi}.
\]
If $\fq \nmid \cN$ then
\[
\Trace(\overline{\rho}_{E,p}(\sigma_\fq)) \equiv a_\fq(E) \pmod{p}.
\]
We note that $\lvert a_\fq(E) \rvert \le 2 \sqrt{\Norm(\fq)}$
by the Hasse--Weil bounds. Moreover, by the injectivity
of torsion, we have $t \mid \# E(\F_\fq)$ (it is here
that we use the assumption that $\fq \nmid t$).
We recall that $\#E(\F_\fq) = \Norm(\fq)+1-a_\fq(E)$.
Hence $\varpi \mid B_{\ff,\fq}$.

Next suppose $\fq \mid \cN$. It follows
from \eqref{eqn:level} that $\fq \mid\mid \cN$,
or equivalently that $E$ has multiplicative
reduction at $\fq$. Thus,
\[
\Trace(\overline{\rho}_{E,p}(\sigma_\fq)) \equiv \pm (\Norm(\fq)+1) \pmod{p},
\]
and therefore $\varpi \mid B_{\ff,\fq}$.
\end{proof}

Lemma \ref{lem:newformelim} is an explicit method 
to discard the isomorphism $\modpg\sim \overline{\rho}_{\mathfrak{f},\varpi}$ for all but a small finite set of primes. 
Namely let  
\[
B_{\mathfrak{f}}=\sum\limits_{\mathfrak{q}\in T}B_{\mathfrak{f}, \mathfrak{q}},
\]
where $T$ is a 
small set of primes $\mathfrak{q}\nmid t \mathcal{N}_{p}$. 
Let $C_{\mathfrak{f}}=\Norm_{\Q_{\mathfrak{f}}/\Q}(B_{\mathfrak{f}})$. 
Then Lemma \ref{lem:newformelim} asserts that $p\mid C_{\mathfrak{f}}$.

We note that if $\ff$ is irrational (i.e. $\Q_\ff \ne \Q$) then
there exists a prime $\fq$ of $\OO_K$ such that $a_{\fq}(\ff)\notin \Q$. 
In this case $B_{\mathfrak{f}, \mathfrak{q}}\neq 0$ 
and it would then be possible to bound $p$. 
Thus, using Lemma~\ref{lem:newformelim}, 
we can eliminate all the irrational $\ff$ (for sufficiently
large $p$),
and we're left with rational $\ff$ which correspond
to elliptic curves $E^\prime$. In that case the 
relation $\overline{\rho}_{E,p} \sim \overline{\rho}_{\ff,\varpi}$

\medskip

We mention that it is sometimes
infeasible 
to compute all Hilbert newforms of parallel weight $2$ 
and level $\mathcal{N}_{p}$ when $\mathcal{N}_{p}$ is 
large (i.e. the norm of $\mathcal{N}_{p}$ is large).
To work around this limitation, one often works with the
characteristic polynomials of Hecke operatons
acting on the space of cusp forms
of parallel weight $2$ and level $\cN_p$, instead of 
computing the individual Hilbert newforms,
as for example in \cite[Section 6]{BPS19},
and \cite[Section 5]{PhilippeFermat}.

\section{Image of inertia and generalisations} 
Let $K$ be a number field. 
Write $G_{K}=\Gal(\bar{K}/K)$ 
for the absolute Galois group of 
$K$. Let $\mathfrak{q}$ be a prime ideal of $\OO_K$. 
Recall that 
\[
D_{\mathfrak{q}}=\{\sigma\in G_{K}\; : \; \sigma(\mathfrak{q})=\mathfrak{q} \}
\]
is the \textbf{decomposition subgroup} of $G_{K}$ 
at $\mathfrak{q}$ and
\[ 
I_{\mathfrak{q}} \; =\; 
\{\sigma\in D_{\mathfrak{q}}\; : \; 
\fq \mid (\sigma(\alpha)-\alpha) \text{ for all $\alpha \in \OO_K$} \}
\]
is the \textbf{inertia subgroup} of $G_{K}$ 
at $\mathfrak{q}$. 

\medskip

The following theorem is a fairly trivial generalization
of ideas of Kraus \cite[Section 6]{Kraus33p}.
\begin{thm}\label{thm:transvection}
Let $K$ be a number field, and let $E$, $E^\prime$
be elliptic curves over $K$. Let 
$\fq$ be a prime of $\OO_K$. Let $p \ge 5$ be a rational
prime. Suppose $E$ has potentially multiplicative reduction
at $\fq$, and moreover, $p \nmid v_\fq(j)$, where
$j$ is the $j$-invariant of $E$.
Suppose $E^\prime$ has potentially additive reduction at $\fq$.
Then $\overline{\rho}_{E,p} \not \thicksim \overline{\rho}_{E^\prime,p}$.
\end{thm}
\begin{proof}
Kraus \cite{Kraus90} determined the possibilities
for the group $\overline{\rho}_{E^\prime,p}(I_\fq)$
for elliptic curves $E^\prime$ with potentially good reduction
at $\fq$.
In particular, $\# \overline{\rho}_{E^\prime,p}(I_\fq) \mid 24$.

However, as $E$ has potentially multiplicative reduction at $\fq$,
and $p \nmid v_\fq(j)$, we know that $p \mid \# \overline{\rho}_{E,p}(I_\fq)$
by Lemma~\ref{lem:transvection}. 
But $p \nmid 24$ as $p \ge 5$. Hence $\overline{\rho}_{E,p} \not \thicksim
\overline{\rho}_{E^\prime,p}$.
\end{proof}

\begin{example}
In \cite{FLTsmall}, the authors
prove Fermat's Last Theorem for several real quadratic
fields. We give some 
of their details for $K=\Q(\sqrt{3})$. 
Suppose $(a,b,c) \in \OO_K^3$
is a non-trivial solution to the Fermat equation $a^p+b^p+c^p=0$
(non-trivial means $abc \ne 0$). As $\OO_K$ has class number $1$,
we may suppse that $\gcd(a,b,c)=1$. Let $E$ be the Frey curve
attached to this solution. It is shown in \cite{FLTsmall}
that $\overline{\rho}_{E,p}$ is irreducible for $p \ge 17$,
and we focus on this case.
It turns out,  after suitably permuting $(a,b,c)$
and scaling by a unit, that the level $\cN_p$
is either $\fq$ or $\fq^4$, where $\fq=(1+\sqrt{3})\OO_K$ 
is the unique prime ideal above $2$ (note $2 \OO_K=\fq^2$). 
There are no newforms at level $\fq$, so we need only
consider $\cN_p=\fq^4$. For the level $\cN_p=\fq^4$
there is only one newform, and this has rational Hecke eigenvalues,
and therefore corresponds to an elliptic curve. This elliptic curve
is
\[
E^\prime \; : \; y^2=x(x+1)(x+8+4\sqrt{3}).
\]
Thus $\overline{\rho}_{E,p} \thicksim \overline{\rho}_{E^\prime,p}$.
However, the $j$-invariant of $E^\prime$ is $j^\prime=54000$;
as $v_\fq(j^\prime) \ge 0$, the elliptic curve $E^\prime$ has
potentially good reduction at $\fq$.

Now we want to study $v_\fq(j)$ where $j$ is the $j$-invariant 
of $E$. 
As $a$, $b$, $c$ are coprime, at most one of them is divisible by
$\fq$. However, $\F_\fq=\OO_K/\fq=\F_2$. If $a$, $b$, $c$
are not divisible by $\fq$, then reducing $a^p+b^p+c^p=0$
modulo $\fq$ gives $1+1+1 \equiv 0$ in $\F_2$ which is a contradiction.
Hence $\fq$ divides exactly one of $a$, $b$, $c$.
For illustration, let's say that $v_\fq(a)=t>0$, and $v_\fq(b)=v_\fq(c)=0$.
Recall,
\[
j=256 \frac{(c^{2p}-a^p b^p)^3}{a^{2p}b^{2p} c^{2p}}.
\]
Thus,
\[
v_\fq(j)=v_\fq(256)-2p t=16-2pt.
\]
Hence, as $p \ge 17$, we have $v_\fq(j)<0$ (i.e. $E$ has
potentially mutliplicative reduction at $\fq$) 
and $p \nmid v_\fq(j)$.
It follows from Theorem~\ref{thm:transvection} that $\overline{\rho}_{E,p}
\not \thicksim \overline{\rho}_{E^\prime,p}$.
\end{example}

For other examples of image of inertia arguments,
see
\cite[Proposition 4.4]{BennettSkinner},
\cite[Section 3]{FS15},
\cite[pages 5--6]{BCDF19}.

\part{The Fermat equation over totally real fields}
Let $n\geq 3$ be an integer. We write $F_n$ for the $n$-th
Fermat curve
\begin{equation}
	\label{eq:Fermat}
F_n \; : \;	x^n+y^n=z^n.
\end{equation} 
The plane curve $F_n$ has genus $(n-1)(n-2)/2$.
Thus, for $n \ge 4$, the genus is $\ge 3$,
and so by Faltings' theorem \cite{Faltings},
$F_n(K)$ is finite for any number field $K$.
Moreover,  
by a theorem of Debarre and Klassen \cite[Theorem 1]{DK94},
for $n\geq 7$, the Fermat curve $F_{n}$ has only 
finitely many points defined over number fields of degree $\leq n-2$. 
We conclude the survey with an overview 
of results on the Fermat equation 
by the modular approach, and other approaches.

\medskip 

\section{Small exponents} 
The third Fermat curve $F_3$ is an elliptic curve, and therefore,
by the Mordell--Weil theorem,
$F_3(K)$ is a finitely generated abelian group.
Although, $F_3(\Q)$ consists of three obvious (trivial) solutions,
for many number fields $F_3(K)$ has positive rank, and is therefore
infinite. Thus, generally one considers the Fermat equation
with $n \ge 4$. To rule out non-trivial solutions for all $n \ge 4$,
it is sufficient to do this for $n=4$, $6$, $9$, and for prime $n\ge 5$.
%
%
Solutions to \eqref{eq:Fermat} are fairly well-understood 
when both $n$ and the degree of $K$ are small; 
we summarise some of the key results in Table \ref{tab:smallexp}.
\begin{table}
\begin{center}
\begin{tabular}{ |c|c|c|c| } 
 \hline
 $n$ & $d$ & References & Results \\
  \hline
  $3$ & $3$ & 
    \makecell{
  Bremner and\\
  Choudhry \cite{Bremner20}
  } & 
\makecell{
Analysis of Galois groups\\
of cubic points on $F_{3}$
}  
  \\
  \hline
  $4$ & $3$, $4$ &
  \makecell{
  Bremner and\\
  Choudhry \cite{Bremner20}
  } & \makecell{Analysis of Galois groups\\ of cubic and quartic points
on $F_4$}\\
  \hline
  $4$, $6$, $9$ & $2$ & Aigner \cite{Ai}, \cite{Aigner57} & 
\makecell{Complete determination of quadratic\\ points on $F_4$, $F_6$ and $F_9$} \\
  \hline
    $4$ & $2$, $3$ & Mordell \cite{Mordell} & 
\makecell{Complete determination of quadratic\\ and
cubic points on $F_4$}\\
   \hline
  $5$ & $\leq 6$ & 
  \makecell{Klassen and\\ 
  Tzermias \cite{Klassen97}} & 
  \makecell{Description of points of degree\\
  $\le 6$ on $F_{5}$}\\
  \hline
  $5$ & $7\leq d\leq 12$ & 
  \makecell{Top and\\
  Sall \cite{Top15}}& 
  \makecell{Description of points of degree\\
  $7 \le d \le 12$ on $F_{5}$}\\
  \hline
  $5$ & $4$ & Kraus \cite{Kraus18} & 
  \makecell{Analysis of Galois groups of\\
    quartic points on $F_{5}$}\\
  \hline
  $5, 7, 11$ & $d\leq (n-1)/2$ & 
  \makecell{Gross and \\
  Rohrlich \cite{Gross78}} & 
  \makecell{Complete determination of degree\\
  $d$ points on $F_{n}$}\\
  \hline
   $13$ & $2$, $3$ & Tzermias \cite{Tzermias04}& 
   \makecell{
Determination of quadratic points on $F_{13}$,\\
and an upper bound for the number \\
of cubic points}\\
  \hline
\end{tabular}
	\caption{This table summarises some key references concerning the Fermat equation \eqref{eq:Fermat} with exponent $n$ over degree $d$ number fields for small $n$ and $d$.}
\end{center}
	\label{tab:smallexp}
\end{table}

\section{Real quadratic fields}

In order to provide a complete resolution of the Fermat equation 
\eqref{eq:Fermat} for all integers $n\geq 4$ over a real quadratic field, by the references in Table \ref{tab:smallexp}, one can assume that the exponent $n=p\geq 17$ is prime.
\medskip

The first generalisation of Wiles' proof of Fermat's last theorem 
over $\Q$ \cite{Wiles} is due to Jarvis and Meekin \cite{Frazer04} who prove that 
there are no non-trivial solutions to \eqref{eq:Fermat} over $\Q(\sqrt{2})$ for all integers $n\geq 4$. 
There are several numerical similarities between the proof of 
Fermat's last theorem over $\Q$ and $\Q(\sqrt{2})$. 
For example, the Frey curve $E$ is semistable over $\Q$ and $\Q(\sqrt{2})$; where 
\[
E:\; Y^2=X(X-a^p)(X+b^p)
\]
is the elliptic curve constructed by Hellegouarch \cite{Hellegouarch}. 
Modularity of the Frey curve $E$ over $\Q(\sqrt{2})$ is a 
direct consequence of prior work of Jarvis and Manoharmayum \cite{JM08}. 
Recall that for $K=\Q$, level-lowering implies the existence of a (classical) newform $f$ of weight $2$ and level $2$ such that $\modpg\sim \overline{\rho}_{f,p}$ -- this leads to a contradiction since there 
are no (classical) newforms of weight $2$ and level $2$. 
Similarly for $K=\Q(\sqrt{2})$, level-lowering asserts the existence of a 
Hilbert newform $\mathfrak{f}$ of parallel weight $2$ and level $\mathfrak{P}$
such that $\modpg\sim \overline{\rho}_{\mathfrak{f},\varpi}$ where $\varpi\mid
p$ and $\mathfrak{P}$ is the unique prime above $2$. There are no Hilbert
newforms of level $\mathfrak{P}$ 
and parallel weight $2$ over $K$.  It is natural to wonder
whether these circumstances are a regular occurence. On the contrary, Jarvis
and Meekin prove that this numerology 
only holds for $K=\Q(\sqrt{2})$ and no other real quadratic field \cite[page 194]{Frazer04}. 

The next breakthrough was due to Freitas and Siksek \cite{FLTsmall} who prove that there are no non-trivial solutions to \eqref{eq:Fermat} over $K=\Q(\sqrt{d})$ for squarefree $3\leq d\leq 26,\; d\neq 5, 17$ for all integers $n\geq 4$. 
Modularity of the Frey curve $E$ over $K$ follows 
immediately from prior work of Freitas, Le Hung and Siksek \cite{quadmod}. 
In order to 
show that $\modpg$ is irreducible for all primes $p\geq 17$ the 
authors use a combination of arguments relating to ray class groups, image of inertia, torsion primes on elliptic curves over number fields of small degree and quadratic points on $X_{0}(34)$. 
In order to eliminate the Hilbert newforms 
resulting from level-lowering the authors 
apply a version of
Lemma \ref{lem:newformelim} -- this yields the bound $p\leq 13$.

This was followed by work of Michaud-Jacobs \cite{PhilippeFermat} 
who proved that there are no non-trivial solutions to the Fermat equation 
\eqref{eq:Fermat} over $K=\Q(\sqrt{d})$ for most squarefree $26\leq d\leq 94$ for all integers $n\geq 4$. 
To deal with the increasing dimension of the 
spaces of Hilbert newforms arising from level-lowering, 
Michaud-Jacobs developed a method of elimination which avoided 
the computation of the full space of Hilbert newforms (see Section \ref{sec:Elimination}).

\section{Totally real cubic fields}
Let $K$ be the totally real cubic field 
with discriminant $148$ or $404$ or $564$. 
Kraus \cite[Theorem 6]{Kraus19} proved that there are 
no non-trivial solutions to \eqref{eq:Fermat} 
over $K$ for prime $p\geq 5$. 
Note that this preceded the work of Derickx, Najman and Siksek 
\cite{cubmod} in which they establish the modularity of elliptic curves over 
totally real cubic fields. 
Instead Kraus applies a criteria of Freitas and Siksek 
(Theorem \ref{thm:somemodular}) to assert the modularity 
of the Frey curve $E$ over $K$.  
Kraus shows that if $\modpg$ is reducible then $E$ has a $K$-rational point of order $p$ or $p\mid D_{K}R_{K}$ where $D_{K}$ is the 
absolute discriminant of $K$ and $R_{K}$ is a computable 
constant depending only on $K$. 
In the first case $p\leq 13$ by Parent's bound \cite{Parent17} and in the 
latter case Kraus uses ray class groups to establish the existence of 
an elliptic curve with a $K$-rational point of order $p$. 
To eliminate the Hilbert newforms arising from level-lowering, Kraus also applies a version of Lemma \ref{lem:newformelim}.


Let $K$ be one of the three cubic fields above.
We show that $F_{4}(K)=F_{4}(\Q)$ using another elementary method. Let 
\[
E^\prime \; : \; Y^2=X(X^2-4).
\]
This is the elliptic curve with Cremona label 
\texttt{64a1}. 
Let $\pi:F_{4}\rightarrow E^\prime$ be 
the map given by 
\[
	\pi: F_{4} \rightarrow E^\prime\quad 
	(x,y,z) \mapsto 
	\left(\frac{z^4}{x^2y^2}, \frac{z^2(x^4-y^4)}{x^3 y^3}\right).
\]
It is straightforward to check using \texttt{Magma} that
\[
E^\prime(K)=E^\prime(\Q)=\{0_{E}, (0,0), (-2,0), (2,0)\}.
\]
Then since 
\[
\pi(F_{4}(K))\subseteq E^\prime(K)=E^\prime(\Q),
\]
it is easy to see that $F_{4}(K)=F_{4}(\Q)$. 
In particular this provides a slightly improved bound on 
the exponent in the previously mentioned result of 
Kraus. 

\section{Higher degree totally real fields}
Let $K=\Q(\zeta_{16})^{+}$ where $\zeta_{16}$ denotes a 
primitive $16^{th}$ root of unity. 
Note that $K$ is a degree $4$ number field 
with Galois group $C_{4}$.
Kraus \cite[Theorem 9]{Kraus19} proved that there are no solutions 
to \eqref{eq:Fermat} over $K$ for 
prime $p\geq 5$. 
The modularity of the Frey curve $E$ over $K$ 
follows immediately from a breakthrough 
result of Thorne \cite[Theorem 1]{Thornemodular}. 
As before, Kraus shows that if $\modpg$ is reducible 
then $p\mid D_{K}R_{K}$ or $E$ has a $K$-rational point of 
order $p$. 
In the first case $p=13$ and in the second it follows 
immediately from \cite{DKSS} that $p\leq 17$. 
To show that $\overline{\rho}_{E,17}$ is irreducible, Kraus 
observes that $X_{0}(17)(K)=X_{0}(17)(\Q)$. 
There are two non-cuspidal points in $X_{0}(17)(\Q)$, corresponding 
to two elliptic curves $E_{1}$ and $E_{2}$ possessing
a $\Q$-rational $17$-isogeny. 
Kraus shows that the $\mathfrak{P}$-adic valuation of $j(E)$ and 
$j(E_{i})$ are distinct for $i=1$, $2$, 
where $\mathfrak{P}$ is the unique prime above $2$. It follows
from this that $\overline{\rho}_{E,p}$ is irreducible, for $p \ge 17$.
The relevant space of Hilbert newforms has dimension $0$, for $p \ge 17$.
It remains to show that 
$\overline{\rho}_{E,13}\not\sim \overline{\rho}_{\mathfrak{f},\varpi}$, where
$\mathfrak{f}$ has level $\mathfrak{P}^r$ for $r\in\{5,6,8\}$ and $\varpi\mid
13$.  To do so, Kraus shows that $a_{\mathfrak{q}}(\mathfrak{f})\not\equiv
a_{\mathfrak{q}}(E)$ (mod $\mathfrak{p}$) where $\mathfrak{q}\mid 79$. 

See \cite{Fermat23} for a resolution of \eqref{eq:Fermat} over $K=\Q(\sqrt{2},\sqrt{3})$ for all integers $n\geq 4$, 
where the degree $4$ number field $K$ has Galois group $V_{4}$.

\nocite{NicolasGithub}
\bibliographystyle{abbrv}
\bibliography{Survey}
\end{document}